\documentclass[12pt]{amsart}
\usepackage[colorlinks=true, pdfstartview=FitV, linkcolor=blue, citecolor=blue, urlcolor=blue]{hyperref}
\usepackage{amssymb,amscd}
\calclayout 

\def\quot#1#2{#1/\!\!/#2}

\def\C{\mathbb {C}}

\def\R{\mathbb {R}}

\def\NN{\mathcal N}

\def\N{\mathbb N}
\def\Z{\mathbb Z}

\def\Hol{\mathcal H}
\def\HH{\mathcal{H}}
\def\SL{\operatorname{SL}}
\def\GL{\operatorname{GL}}

\def\PSL{\operatorname{PSL}}

\def\Sp{\operatorname{Sp}}
\def\Spin{\operatorname{Spin}}
\def\SO{\operatorname{SO}}
\def\Orth{\operatorname{O}}

\def\SU{\operatorname{SU}}

\def\Diffeo{\operatorname{Diff}}
\def\Diff{\mathcal D}
\def\HDiff{\mathcal {HD}}

\def\inv{^{-1}}
\def\lie#1{{\mathfrak #1}}
\def\lieg{\lie g}
\def\lieh{\lie h}

\def\liesl{\lie {sl}}

\def\liek{\lie {k}}

\def\liez{\lie{z}}
\def\phi{{\varphi}}

\def\End{\operatorname{End}}
\def\AA{\mathsf{A}}
\def\CC{\mathsf{C}}
\def\BB{\mathsf{B}}
\def\DD{\mathsf{D}}

\def\GG{\mathsf{G}}

\def\E{\mathcal E}
\def\O{\mathcal O}
\def\M{\mathcal M}
\def\F{\mathcal F}
\def\LL{\mathcal L}

\def\pr{{\operatorname{pr}}}

\def\Ker{\operatorname{Ker}}

\def\Aut{\operatorname{Aut}}
\def\HAut{\Aut_{\HH}}

\def\codim{\operatorname{codim}}

\def\rank{\operatorname{rank}}

\def\lieA{{\lie A}}

\def\Sym{\operatorname{S}}
\def\ql{{\operatorname{{q\ell}}}}
\def\HHH{\operatorname{H}}
\def\G{\mathcal G}

\numberwithin{equation}{subsection}

\newtheorem{theorem}[subsection]{Theorem}
\newtheorem{lemma}[subsection]{Lemma}
\newtheorem{proposition}[subsection]{Proposition}
\newtheorem{corollary}[subsection]{Corollary}

\theoremstyle{definition}

\newtheorem{definition}[subsection]{Definition}

\newtheorem{problem}[subsection]{Problem}

\theoremstyle{remark}

\newtheorem{remark}[subsection]{Remark}
\newtheorem{remarks}[subsection]{Remarks}
\newtheorem{example}[subsection]{Example}

\title[Quotients, automorphisms and differential operators]{\boldmath Quotients, automorphisms and differential operators} 
 \author{Gerald W. Schwarz}
\address{Department of Mathematics\\
Brandeis University\\
Waltham, MA 02454-9110}
\email{schwarz@brandeis.edu}

\begin{document}
\begin{abstract}
Let $V$ be a $G$-module where $G$ is a complex reductive group. Let $Z:=\quot VG$ denote the categorical quotient and let $\pi\colon V\to Z$ be the morphism dual to the inclusion $\O(V)^G\subset\O(V)$.    Let $\phi\colon Z\to Z$ be an algebraic  automorphism. Then one can ask if there is an algebraic map $\Phi\colon V\to V$ which lifts $\phi$, i.e., $\pi(\Phi(v))=\phi(\pi(v))$ for all $v\in V$.  In Kuttler \cite{Kuttler}   the case is treated where $V=r\lieg$ is a multiple of the adjoint representation of $G$. It is shown that, for $r$ sufficiently large (often $r\geq 2$ will do), any $\phi$ has a lift. 

We consider the case of general representations (satisfying some mild assumptions). It turns out that it is natural to consider holomorphic lifting of holomorphic automorphisms of $Z$, and we show that if a holomorphic $\phi$ and its inverse  lift  holomorphically, then $\phi$ has a lift $\Phi$ which is an automorphism such that $\Phi(gv)=\sigma(g)\Phi(v)$, $v\in V$, $g\in G$ where $\sigma$ is  an  automorphism of $G$. We reduce the lifting problem to the group of automorphisms of $Z$ which preserve the natural  grading of $\O(Z)\simeq\O(V)^G$.
Lifting does not always hold, but we show that it always does for representations of tori in which case algebraic automorphisms lift to algebraic automorphisms.  We extend Kuttler's methods to show lifting in case $V$ contains a copy of $\lieg$.  
 \end{abstract}

\maketitle

\section{Introduction}\label{sec:intro}
Our base field is $\C$, the field of complex numbers. Let $G$ be a complex reductive group and $V$  a  $G$-module. We denote the algebra of polynomial functions on $V$ by $\O(V)$. For the following, we refer to \cite{KraftBook}, \cite{LunaSlice} and \cite{PopovVinberg}. By Hilbert, the algebra $\O(V)^G$ is finitely generated, so that we have a quotient variety $Z:=\quot VG$ with coordinate ring $\O(Z)=\O(V)^G$. Let $\pi\colon V\to Z$ denote the morphism dual to the inclusion $\O(V)^G\subset\O(V)$. Then $\pi$ sets up a bijection between the points of $Z$ and the closed orbits in $V$. If $Gv$ is a closed orbit, then the isotropy group $H=G_v$ is reductive. The \emph{slice representation of $H$ at $v$\/} is its action on $N_v$ where $N_v$ is an $H$-complement to $T_v(Gv)$ in $T_v(V)\simeq V$. Let $Z_{(H)}$ denote the points of $Z$ such that the isotropy groups of the corresponding closed orbits are in the conjugacy class $(H)$ of $H$.  The $Z_{(H)}$ give a finite stratification of $Z$ by locally closed smooth subvarieties. In particular, there is a unique open stratum $Z_{(H)}$, the \emph{principal stratum}, which we also denote by $Z_\pr$. We call   $H$ a \emph{principal isotropy group} and any associated closed orbit a \emph{principal orbit\/} of $G$.   

As shorthand for saying that $V$ has finite (resp.\ trivial)  principal isotropy groups we say that $V$ has FPIG (resp.\ TPIG). If $V$ has FPIG, then  there is an open set of closed orbits and a closed orbit is principal if and only if the slice representation of its isotropy group is trivial. Set $V_\pr:=\pi\inv(Z_\pr)$.  We say that $V$ is \emph{$k$-principal\/} if $V$ has FPIG and $\codim V\setminus V_\pr\geq k$.

Let $\Diff^k(V)$ (resp.\ $\Diff^k(Z)$) denote the differential operators on $V$ (resp.\  $Z$) of order at most $k$ (see \cite[\S 3]{SchLiftingDOs}). Then restriction gives us a morphism $\pi_*\colon\Diff^k(V)^G\to\Diff^k(Z)$. One just considers  elements of  $\Diff^k(V)^G$ as differential operators on $\O(V)^G=\O(Z)$. 

\begin{definition}\label{def:admissible}
We say that $V$ is \emph{admissible\/} if 
\begin{enumerate}
\item $V$ is $2$-principal.\label{def:2princ}
\item  $\pi_*\colon\Diff^k(V)^G\to\Diff^k(Z)$ is surjective for all $k$.  \label{def:pistarsurj}
\end{enumerate}
\end{definition}
If $V$ is  $2$-principal, then the principal isotropy group is the kernel of $G\to\GL(V)$ \cite[Remark 2.5]{SchVectorFields}.  Thus if $V$ is admissible, we can assume that it has TPIG by just dividing out by the ineffective part of the action. If $G^0$ is a torus, then  
(1) implies (2) (see \cite[10.4]{SchLiftingDOs}). 

Let $\lieg$ denote  the Lie algebra of $G$. Suppose that $G=G_1\times G_2$ is a product of reductive groups. Consider representations $V=V_1\oplus V_2$ where $V_i$ is a representation of $G_i$, $i=1$, $2$. If $V_1$ is not admissible, then neither is $V_1\oplus V_2$. This explains why the hypotheses of the following theorem are necessary.

\begin{theorem}(\cite[Corollary 11.6]{SchLiftingDOs})
\label{thm:finitebad}
Let $G$ be a connected semisimple group and consider representations of $G$ which contain no trivial factor and all of whose irreducible factors are faithful representations of $\lieg$. Then, up to isomorphism, there are only finitely many such representations which are not admissible.  
\end{theorem}

For $X$ an affine variety let $\Aut(X)$ denote the automorphisms of $X$ and let $\HAut(X)$ denote the holomorphic automorphisms of $X$.   Let $\phi\in\HAut(Z)$. 
We say that \emph{$\phi$ preserves the stratification\/} if it permutes the strata. From \cite[Theorems 1.1 and 1.3, Remark 2.4]{SchVectorFields} we have the following result.

\begin{theorem}\label{thm:stratapreserve}
Let $V$ be admissible, $3$-principal or $2$-principal and orthogonal. Then any $\phi\in\HAut(Z)$ preserves the stratification of $Z$.
\end{theorem}

Let $\phi\in\HAut(Z)$. We say that  a holomorphic map $\Phi\colon V\to V$ is a \emph{lift of $\phi$\/}  if $\pi\circ\Phi=\phi\circ\pi$. Equivalently,  $\Phi$   maps the fiber   $\pi\inv(z)$  to the fiber $\pi\inv(\phi(z))$ for every $z\in Z$. 

\begin{remark}
Suppose that $\phi\in\HAut(Z)$ preserves the stratification. Then $\phi$ preserves the principal stratum (the one of largest dimension) and the fixed points $V^G\subset Z$ (the stratum of lowest dimension). Now $V$ is a direct sum $V^G\oplus V'$ of $G$-modules, and   the restriction of $\phi$ to $V^G$ extends to an automorphism $\Phi$ of $V$ which sends $(v,v')$ to $(\phi(v),v')$  where $v\in V^G$, $v'\in V'$. Thus we may always reduce the problem of lifting $\phi$ to the case that $\phi(\pi(0))=\pi(0)$. The same argument applies if we are looking for algebraic lifts in case that $\phi\in\Aut(Z)$.
\end{remark}

Let $\Phi$ be a lift of $\phi$ and let $\sigma$ be an automorphism of $G$. We say that $\Phi$ is \emph{$\sigma$-equivariant\/} if $\Phi(gv)=\sigma(g)\Phi(v)$ for all $v\in V$, $g\in G$. Let $\Aut_\ql(Z)$ denote the set of \emph{quasilinear automorphisms\/} of $Z$, i.e., the automorphisms of $Z$ which preserve the grading of $\O(Z)\simeq\O(V)^G$. This is a linear  algebraic group. If $\Phi $ is a holomorphic lift of $\phi\in\Aut_\ql(Z)$ with $\Phi(0)=0$, then so is $\Phi'(0)$, hence
$\phi$ has a linear lift.  When $V$ is $2$-principal and $\phi$ preserves the principal stratum we will see that $\Phi'(0)$ is invertible and normalizes $G$,  hence induces an automorphism $\sigma$ of $G$.

\begin{theorem}\label{thm:intro:reducequasilinear}
Let $V$ be an admissible $G$-module and let $\phi\in\HAut(Z)$ where $\phi(\pi(0))=\pi(0)$. Then there is a holomorphic family $\phi_t\in\HAut(Z)$, $t\in\C$,  with $\phi_0\in\Aut_\ql(Z)$ and $\phi_1=\phi$. 
\end{theorem}

\begin{theorem}\label{thm:intro:lifting}
Let $V$ be admissible, $3$-principal or $2$-principal and orthogonal. Assume that we have a holomorphic family $\phi_t \in\HAut(Z)$, $t\in\C$, where $\phi_0\in\Aut_\ql(Z)$ and $\phi_1=\phi$. Then there are lifts $\Psi_t\in\HAut(V)^G$ of $\phi_0\inv\phi_t$, $t\in[0,1]$. If $\phi_0$ has a linear lift to $V$, then  $\phi_0$ has a $\sigma$-equivariant linear lift $\Phi_0\in\GL(V)$  for some $\sigma$ and $\Phi_0\Psi_1$ is a $\sigma$-equivariant biholomorphic  lift of $\phi$. 
Finally, any element of $\Aut_\ql(Z)^0$ has a lift to $\GL(V)$.
\end{theorem}

 \begin{remark} \label{rem:alglifting}
 In case that $G^0$ is a torus, Theorem \ref{thm:intro:liftingtori} below shows that algebraic automorphisms of $Z$ lift to $\sigma$-equivariant automorphisms of $V$. 
We doubt that this is usually the case when $G^0$ is not a torus.     
\end{remark}
 
 \begin{remark}
 In case $G$ is finite, the lifting problem has been considered in   \cite{GottschlingInvarianten}, \cite{KrieglTensor}, \cite{PopovMichor} and \cite{PrillLocal}.
 \end{remark}
 
 \begin{remark}\label{rem:lifting} Suppose that $\phi\in\HAut(Z)$, $\phi(\pi(0))=\pi(0)$, and that $\phi$ and $\phi\inv$ lift to $V$ where $V$ is admissible, $3$-principal or $2$-principal and orthogonal. Then there is a holomorphic family $\phi_t$ as above such that  $\phi_0\in\Aut_\ql(Z)$ has a  linear lift (Proposition  \ref{prop:liftsimpliesdiff}  below).  Theorem \ref{thm:intro:lifting} shows that $\phi$ has  a biholomorphic $\sigma$-equivariant lift. Thus $\phi$ sends the   stratum with conjugacy class  $(H)$ to the stratum with conjugacy class  $(\sigma(H))$, i.e., $\phi$ permutes the strata of $Z$ according to the action of $\sigma$ on conjugacy classes. Kuttler \cite{Kuttler} has a similar result for $\phi$ algebraic.
\end{remark}

 From the theorems above  it is clear that the lifting problem reduces to one for the connected components of $\Aut_\ql(Z)$. We say that $V$ has the \emph{lifting property\/} if every element of $\Aut_\ql(Z)$ lifts to $\GL(V)$. In the cases of the classical representations of the classical groups (and a bit more) we can say exactly what happens. In the theorem below the actions of $\GG_2$ (resp.\  $\BB_3=\Spin_7$) on $\C^7$ (resp.\ $\C^8$)   are the irreducible representations of the given dimension.
  
 \begin{theorem}\label{thm:citlifting}
 Let $(V,G)$ be one of the following representations. Then $V$ is admissible and $V$ has the lifting property.
 \begin{enumerate}
\item $(k\C^n\oplus\ell(\C^n)^*,\SL_n)$, $n\geq 3$ where $k+\ell\geq 2n$ and if $k\ell=0$, then $k+\ell> 2n$.\item $(k\C^n\oplus\ell(\C^n)^*,\GL_n)$ where   $k$,  $\ell>n$.
\item $(k\C^n,\SO_n)$ where $k\geq n\geq 3$. 
\item $(k\C^n,\Orth_n)$ where $k>n\geq 3$.
\item $(k\C^{2n},\Sp_{2n})$ where $k\geq 2n+2$, $n\geq 2$.
\item $(k\C^7,\GG_2)$ where $k>4$.
\item $(k\C^8,\BB_3)$ where $k>5$.
\end{enumerate}
 \end{theorem}
 
 \begin{remark}\label{rem:nolift}
 The conditions in the theorem are mostly those needed to guarantee that the representations are admissible (see \cite[\S 11]{SchLiftingDOs}). However, $(V,G)$ is admissible and we have an element of $\Aut_\ql(Z)$ which does not lift in the following cases.
 \begin{enumerate}
\item $(2n\C^n,\SL_n)$ and $(2n(\C^n)^*,\SL_n)$, $n\geq 2$.
\item $(4\C^7,\GG_2)$.
\item $(5\C^8,\BB_3)$.
\end{enumerate}
 \end{remark}
 
If $G^0$ is a torus  there is no problem in lifting.
 
 \begin{theorem}\label{thm:intro:liftingtori}
 Suppose that $V$ is $2$-principal and that $G^0$ is a torus. Then every algebraic (resp.\ holomorphic) automorphism of $Z$ lifts to an algebraic (resp.\ holomorphic) automorphism of $V$ which is $\sigma$-equivariant for some $\sigma$.
 \end{theorem}
 
 Kurth\cite{Kurth}   considered the lifting  problem with $G$ as above  for connected groups of automorphisms of the quotient.
 
 We now consider a strengthening of Theorem \ref{thm:intro:liftingtori}. Let $G$ be reductive and write $G^0=G_sS$ where $G_s$ is semisimple and $S$ is the connected center of  $G^0$. Assume that  $V$ is $4$-principal and let $\phi\in\Aut(Z)$. Then (see Proposition \ref{prop:reducess}) $\phi$ lifts to an automorphism of $Z_s:=\quot V{G_s}$, so the lifting problem rests entirely in the action of $G_s$.  
  
 Using ideas of Kuttler \cite{Kuttler} we establish the following result.
 
 \begin{theorem}\label{thm:genkuttler}
Suppose that  $V$ is $4$-principal and
   $(V,G_s)$ contains a copy of $\lieg_s$. Then any $\phi\in\Aut(Z)$ has a lift to $V$. 
 \end{theorem}
 
 Note that we don't claim that the lifting is $\sigma$-equivariant or an automorphism. By Remark \ref{rem:lifting}, however,  there is a $\sigma$-equivariant biholomorphic lift and $\phi$ permutes the strata via the action of $\sigma$ on conjugacy classes.

 The following result is quite surprising. It says that 
 $Z$ often determines $V$ and $G$.  

 \begin{corollary}\label{cor:surprising}
Let $G_i\subset  \GL(V_i)$ where each $(V_i,G_i)$ satisfies the conditions of Theorem \ref{thm:genkuttler}.    Let $Z_i$ denote $\quot {V_i}{G_i}$ and suppose that there is an algebraic isomorphism $\psi\colon Z_1\to Z_2$. Then there is a linear isomorphism $V_1\simeq V_2$ inducing an isomorphism of $G_1$ with $G_2$.
\end{corollary}
 
 \begin{proof}
 Let $V=V_1\oplus V_2$ and $G=G_1\times G_2$. Then $\quot VG\simeq Z_1\times Z_2$ and     $(V,G)$ satisfies the hypotheses of Theorem \ref{thm:genkuttler}. We have an automorphism $\phi$ of $Z_1\times Z_2$ which sends $(z_1,z_2)$ to $(\psi\inv(z_2),\psi(z_1))$. Since $\phi$ induces an automorphism of $V_1^{G_1}\times V_2^{G_2}$, $\psi$ sends $V_1^{G_1}$ isomorphically onto $V_2^{G_2}$ and we may arrange that $\psi$ sends $0\in V_1^{G_1}$ to $0\in V_2^{G_2}$. Then $\phi$ preserves $\pi(0)$. By Theorem  \ref{thm:genkuttler} and Remark \ref{rem:lifting},   $\phi$  lifts to a holomorphic automorphism $\Phi\colon V\to V$ such that $\Phi(gv)=\sigma(g)\Phi(v)$ for $g\in G$, $v\in V$ where $\sigma$ is an automorphism of $G_1\times G_2$. Since $\Phi$ sends closed orbits to closed orbits and  lies over $\phi$, which permutes the strata, $\Phi$ sends $(V_1 )_\pr \times(0)$  isomorphically onto  $(0)\times (V_2)_\pr$, hence it sends $V_1\times(0)$ isomorphically onto $(0)\times V_2$   and clearly $\sigma$ has to interchange the factors of $G_1\times G_2$. Now $\Phi'(0)$ induces a $\sigma$-equivariant linear isomorphism of $V_1$ with $V_2$. 
  \end{proof}

\begin{remark} Of course, a similar corollary holds for $2$-principal actions where $G^0$ is a torus. In this case we get the desired result even if there is only a holomorphic isomorphism of $Z_1$ and $Z_2$.
\end{remark}
  
We generalize the results of Kuttler for  multiples of  the adjoint representation of a semisimple group and obtain a best possible result.  Let $V=\oplus_{i=1}^d r_i\lieg_i$ where the $\lieg_i$ are  simple Lie algebras. Let $G_i$ denote the adjoint group of $\lieg_i$ and let $G$ denote the product of the $G_i$. We assume that $r_i\geq 2$ for all $i$ and that $r_i\geq 3$ if $\lieg_i\simeq\liesl_2$.   These conditions are necessary and sufficient for $(V,G)$ to be $2$-principal and for any $\phi\in\Aut(Z)$ to be strata preserving  \cite[Theorem 1.2, Proposition 3.1]{SchVectorFields}. 
\begin{theorem}\label{thm:adjoint}
Let $V=\oplus_{i=1}^d r_i\lieg_i$  be as above. Then any $\phi\in\Aut(Z)$ lifts to $V$.
\end{theorem}

 The $G$-module $V$ is not $4$-principal if any $r_i\lieg_i$ is $3\liesl_2$, $4\liesl_2$,  $2\liesl_3$,   $2\lie{so}_5$ or $2\liesl_4$. We have to carefully analyze these cases in order to reduce to the arguments of Theorem  \ref{thm:genkuttler}. Again, the lifting is not necessarily $\sigma$-equivariant or an automorphism. In \cite{SchAdjoint} we prove that $V$ is admissible, hence elements of $\HAut(Z)$ have $\sigma$-equivariant lifts to $\HAut(V)$.

Our results lead to consequences for actions of compact Lie groups   (see \cite{SchLifting}). Let $K$ be a compact Lie group  and $W$ a real $K$-module. Let $W/K$ be the orbit space where we say that a function $f\colon W/K\to \R$ is $C^\infty$ if it pulls back to an element of $C^\infty(W)$ (necessarily $K$-invariant). Then we can define the notion of a smooth automorphism of $W/K$. (See \S \ref{sec:compact} for more details.)\ Let $G=K_\C$ be the complexification of $K$ and set $V=W\otimes_\R\C$. We say that $W$ is \emph{admissible\/} if $V$ is $2$-principal.

\begin{theorem}\label{thm:compact}
Let $W$ be an admissible $K$-module  and let $\phi\colon W/K\to W/K$ be a smooth automorphism. If $V$ has the lifting property, then there is a   lift $\Phi\in \Diffeo(W)$ of $\phi$ and a group automorphism $\sigma$ of $K$ such that $\Phi(kw)=\sigma(k)\Phi(w)$ for every $k\in K$ and $w\in W$. 
\end{theorem}

It is possible that $W$ has the lifting property (obvious definition) even if $V$ does not, but we don't know of an example.
Related results on differentiable lifting  can be found in \cite{BierstoneLifting}, \cite{LosikLift}, \cite{SchLifting} and \cite{StrubLocal}.

 Here is an outline of the paper. In \S \ref{sec:liftableautoms} we consider when an automorphism $\phi$ of $Z$ can be deformed to a quasilinear automorphism and prove Theorem \ref{thm:intro:reducequasilinear}. In \S \ref{sec:liftingcomplex} we establish Theorem \ref{thm:intro:lifting} on holomorphic equivariant liftings. In \S \ref{sec:examples} we establish Theorem \ref{thm:citlifting} and Remark \ref{rem:nolift} (our examples).  In \S \ref{sec:torus} we establish Theorem \ref{thm:intro:liftingtori}, the result on actions of tori,  and in \S \ref{sec:kuttler} we establish Theorem  \ref{thm:genkuttler} on lifting for representations containing $\lieg$.  In \S\ref{sec:adjoint}  we consider the generalized adjoint case and establish Theorem \ref{thm:adjoint} and in \S \ref{sec:compact} we establish Theorem \ref{thm:compact} on the actions of compact groups.

\medskip
\noindent {\sc Acknowledgement}

\medskip
We thank the members of the Mathematics Department of the Ruhr-Universit\"at Bochum for their hospitality while this research was being performed. We also thank the referee for a careful reading of our paper which has led to numerous improvements and corrections.

    \section{Deformable and quasilinear automorphisms}\label{sec:liftableautoms}
Let   $V$ be a $G$-module. We have a $\C^*$-action on $\quot VG$ where $t\cdot\pi(v)=\pi(tv)$, $t\in\C^*$, $v\in V$. One can also see this   as follows.  Let $p_1,\dots,p_d$ be homogeneous generators of $\O(V)^G$ and let $p=(p_1,\dots,p_d)\colon V\to\C^d$. Let $Y$ denote the image of $p$. Then we can identify $Y$ with $Z=\quot VG$. If $e_i$ is the degree of $p_i$, then we have a $\C^*$-action on $Y$ where $t\in\C^*$ sends  $(y_1,\dots,y_d)\in Y$ to $(t^{e_1}y_1,\dots,t^{e_d}y_d)$. The isomorphism $Y\simeq Z$ identifies the two $\C^*$-actions.
  
  Suppose that $\phi\in\HAut(Z)$ and $\phi(\pi(0))=\pi(0)$. We say that $\phi$ is \emph{deformable\/} if $\phi_0(z):=\lim_{t\to0}t\inv\cdot\phi(t\cdot z)$ exists for every $z\in Z$. Note that if all the $p_i$ have the same degree, then the limit exists and is just an ordinary derivative.  
 
Let $S=\bigoplus S_k$ denote the graded ring $\O(Z)$ and let $F_k=S_k+S_{k+1}+\dots$. . Let $\Hol(Z)$ (resp.\ $\Hol(V)$) denote the holomorphic functions on $Z$ (resp.\ $V$).  

\begin{remarks}\label{rem:smoothphit}
\begin{enumerate}
\item Let $\phi\in\Aut(Z)$ where $\phi(\pi(0))=\pi(0)$. Then  $\phi$ is deformable if and only if $\phi^*S_k\subset F_k$ for all $k\geq 0$. In this case the family $\phi_t$, $t\in\C^*$,  extends to an algebraic  family $\phi_t$, $t\in\C$. 
\item Let $\phi\in\HAut(Z)$  where $\phi(\pi(0))=\pi(0)$. Then $\phi$ is deformable if and only if $\phi^*S_k\subset S_k\cdot\Hol(Z)$ for all $k\geq 0$. In this case the family $\phi_t$, $t\in\C^*$, extends to a holomorphic family $\phi_t$, $t\in\C$. 
\end{enumerate}
\end{remarks}

We say that a differential operator $P\in\Diff(Z)$ is \emph{homogeneous of degree $\ell$\/} if it sends elements of $S_m$ to $S_{m+\ell}$ for all $m$. An automorphism $\phi\in\Aut(Z)$ acts on $\Diff^k(Z)$ sending $P$ to $\phi_*(P):=(\phi\inv)^*\circ P\circ\phi^*$. If $\phi\in\HAut(Z)$, then $\phi$ acts on the
 holomorphic differential operators $\HDiff^k(Z)$ of order $k$ on $Z$ by the same formula. 
 If $V$ is admissible, we have a surjection 
  $$
  \HDiff^k(V)^G=\Hol(V)^G\otimes_{\O(V)^G}\Diff^k(V)^G\to \Hol(Z)\otimes_{\O(Z)}\Diff^k(Z)=\HDiff^k(Z).
  $$
  
The following implies Theorem \ref{thm:intro:reducequasilinear}.

\begin{theorem}\label{thm:admissibleimpliesdiff}
Let $V$ be admissible and let $\phi\in\HAut(Z)$ where $\phi(\pi(0))=\pi(0)$. Then $\phi$ is deformable and $\phi_0\in\Aut_\ql(Z)$.
\end{theorem}

\begin{proof}    
Let $f\in S_k$ and suppose that $\phi^*f=h_\ell+F_{\ell+1}\HH(Z)$ where $\ell<k$, $0\neq h_\ell\in S_\ell$. Let $P\in \O(V^*)^G$ be dual to $\pi^*h_\ell$ under a choice of basis for $V$ (and hence of $\O(V)$ and $\O(V^*)$). We may consider $P$ as   an invariant constant coefficient differential operator  of  order $\ell$. Then $P(\pi^*h_\ell)$ is a nonzero constant and we can arrange that $P(\pi^*h_\ell)=1$. Now $\phi^*((\phi_*\pi_*P)(f))=(\pi_*P)(\phi^*f)$ and $(\pi_*P)(\phi^*f)=1$ at $\pi(0)$, by construction.  Since $\phi_*\pi_*P$ has order at most $\ell$ (actually $\ell$ by \cite[Corollary 5.10]{SchLiftingDOs}),  it equals $\pi_*Q$ where   $Q=\sum h_iQ_i$ and the $h_i$ are in $\Hol(V)^G$ and the $Q_i$ are in $\Diff^\ell(V)^G$. But each $Q_i$  is a sum of terms of degree at least $-\ell$. Hence   $\pi^*(\phi_*\pi_*P(f))=\sum_i h_iQ_i(\pi^*(f))$ where $Q_i(\pi^*(f))$ is a sum of terms of degree at least $k-\ell>0$. Thus $Q_i(\pi^*(f))$ vanishes at $0$ and   $(\phi_*\pi_*P)(f)$ vanishes at $\pi(0)$ which is a contradiction. Thus $\phi$ is deformable. Since $\phi\inv$ is also deformable, we see that $\phi_0$ is invertible, hence $\phi_0\in\Aut_\ql(Z)$.
\end{proof}

\begin{corollary}\label{cor:gradedisomorphism}
Let $G_i\subset\GL(V_i)$, $i=1$, $2$. Assume that the $V_i$ are admissible. Set $Z_i=\quot {V_i}{G_i}$ and suppose that there is a holomorphic isomorphism $\psi\colon Z_1\to Z_2$. Then there is an algebraic  isomorphism preserving the gradings of $\O(Z_1)$ and $\O(Z_2)$.
\end{corollary}

\begin{proof}
Let $V=V_1\oplus V_2$ and $G=G_1\times G_2$. Then $Z:=\quot VG\simeq Z_1\times Z_2$.   Let $\phi\colon Z\to Z$ send $(z_1,z_2)$ to $(\psi\inv(z_2),\psi(z_1))$ for $z_1\in Z_1$, $z_2\in Z_2$. As in the proof of Corollary \ref{cor:surprising},  we may assume  that $\phi$ fixes the image of $0\in V_1\oplus V_2$. Then $\phi$ is deformable, hence there is a graded isomorphism of $\O(Z_1)\otimes\O(Z_2)$ which interchanges the two factors.
\end{proof}

\begin{example}\label{ex:notadmissible}
 Let $(V_1,G_1):=(2n\C^{n-1},\SL_{n-1})$ and $(V_2,G_2):=(2n\C^{n+1},\SL_{n+1})$ where $n\geq 4$.   The quotient $Z_1$ is isomorphic to the decomposable elements  in $\wedge^{n-1}(\C^{2n})$ while $Z_2$ is isomorphic to the decomposable elements of $\wedge^{n+1}(\C^{2n})$. We have a canonical isomorphism $Z_1\simeq Z_2$. The generators of $\O(V_1)^{G_1}$ have degree $n-1$ while those of $\O(V_2)^{G_2}$ have degree $n+1$. The representation $V_1$ is admissible while $V_2$ is not  (but almost!).  Both representations are $4$-principal.  Now consider $(V,G)=(V_1\oplus V_2,G_1\times G_2)$. Then the quotient is $Z_1\times Z_2$ and we have an automorphism $\phi$ which interchanges the $Z_i$.  It clearly is not deformable because of the difference in degrees of the generators.   Thus Theorem \ref{thm:admissibleimpliesdiff} can fail if the representation is not admissible and Corollary \ref{cor:gradedisomorphism} can fail if one of the representations in question is not admissible.  Also, Theorem \ref{thm:genkuttler}  and Corollary \ref{cor:surprising} can fail when  $\lieg_s$ is not a subrepresentation of $V$.
  \end{example}

   Let $N_{\GL(V)}(G)$ denote the normalizer of $G$ in $\GL(V)$.
   
 \begin{proposition}\label{prop:normalizer}
 Let $\phi\in\Aut_\ql(Z)$ and suppose that $\phi$ lifts to $\Phi\in\End(V)$ where $V$ is $2$-principal and $\phi$ preserves the principal stratum of $Z$. Then 
 \begin{enumerate}
\item There is a morphism $\sigma\colon G\to G$ such that $\Phi(gv)=\sigma(g)\Phi(v)$, $g\in G$, $v\in V$.
\item $\Phi\in N_{\GL(V)}(G)$.
\item $\sigma(g)=\Phi\circ g\circ \Phi\inv$, $g\in G$,  so that $\sigma\in\Aut(G)$. 
\item If $\phi$ is the identity, then $\Phi$ is multiplication by an element $g\in G$.
\end{enumerate}
 \end{proposition}
 
 \begin{proof}
 We know that $\Phi$ sends the fiber of $\pi$ over $z\in Z$ to the fiber over $\phi(z)$. Consider a fiber which is a principal orbit $Gv$. Then $\Phi$ sends $Gv$ to the principal orbit $G\Phi(v)$. Thus $\Phi(gv)=\Psi(v,g) \Phi(v)$ where $\Psi(v,g)\in G$ is unique.  Now  $V_\pr\times_{Z_\pr}V_\pr\simeq V_\pr\times G$ so that  $\Psi\colon V_\pr\times G\to G$ is a morphism. Since $V$ is $2$-principal, $\Psi$ extends to a morphism of $V\times G$ to $G$.  By construction, $\Psi(v,g)=\Psi(tv,g)$ for any $t\in\C^*$. Hence $\Psi(v,g)=\sigma(g):=\Psi(0,g)$ for all principal $v$. It follows that   $\Phi(gv)=\sigma(g) \Phi(v)$  for all $v\in V$ and we have (1).  
 
It follows from (1) that  $\Ker\Phi$ is $G$-stable. Write $V=\Ker\Phi\oplus V'$ where $V'$ is $G$-stable. Let $v_1+v_2\in\Ker\Phi\oplus V'$. Then $v_1+v_2\in V_\pr$ if and only if $\Phi(v_2)\in V_\pr$ if and only if $v_2\in V_\pr$, in which case the principal orbits $G(v_1+v_2)$ and $Gv_2$ coincide. Since   $Gv_2\subset V'$ this forces $v_1=0$. Hence $\Ker\Phi=0$ and   $\Phi\in\GL(V)$.  Now for $v$ principal we have that $\Phi(gv)=(\Phi\circ g\circ \Phi\inv)\Phi(v)$, $g\in G$, so that $\sigma(g)=\Phi\circ g\circ\Phi\inv$. Hence we have (2) and (3).

Suppose that $\phi$ is the identity. Then for $v$ principal there is a unique $\tau(v)\in G$ such that $\Phi(v)=\tau(v)v$. Arguing as above, $\tau$ extends to a morphism $V\to G$ and  $\Phi(v)=\tau(0)v$ for all $v\in V$ and we have (4).
  \end{proof}

\begin{remark}\label{rem:conjbyg} Let $v\in V$ and $g\in G$.
We may change $\Phi$ to $\Phi_g$ where $\Phi_g(v)=\Phi(gv)$. 
Then for $h\in G$ we have
$$
\Phi_g(hv)=\Phi(ghg\inv  gv)=\sigma(g)\sigma(h)\sigma(g)\inv\Phi_g(v).
$$
Thus we  change  $\sigma$ by an inner automorphism of $G$. Hence if $\sigma$ is inner, we can arrange that $\Phi$ is $G$-equivariant. In general, we can arrange that $\sigma$ is a diagram automorphism (if $G$ is semisimple).
\end{remark}

\begin{corollary}
Let $V$ be admissible, $3$-principal or $2$-principal and orthogonal. Then $V$ has the lifting property  if and only if   $N_{\GL(V)}(G)$ maps onto $\Aut_\ql(Z)$.
\end{corollary}
 
 \begin{proposition}\label{prop:autql0}
 Let $V$ be admissible, $3$-principal or  $2$-principal and orthogonal. Then  $\GL(V)^G$ maps onto  $\Aut_\ql(Z)^0$.
 \end{proposition}
 
 \begin{proof}
 Under the hypothesis,  $\pi_*\colon\Diff^1(V)^G\to\Diff^1(Z)$ is surjective \cite[Theorem 1.3]{SchVectorFields}. Since $\pi_*$ preserves the degrees of   differential operators it  sends the Lie algebra of $\GL(V)^G$ which is   $\End(V)^G$ (the degree zero derivations)  onto the Lie algebra of degree zero derivations of $\O(Z)$ which is the Lie algebra of $\Aut_\ql(Z)$.
  \end{proof}

We end this section with the observation that if  $\phi$ and $\phi\inv$ have lifts, then $\phi$ is deformable.

\begin{proposition}\label{prop:liftsimpliesdiff}
Suppose that $\phi\in\HAut(Z)$ and $\phi(\pi(0))=\pi(0)$ where $\phi$ is stratification preserving and $V$ is $2$-principal.  
Suppose that $\phi$ lifts to $\Phi$ and that $\phi\inv$ also lifts.
Then $\phi$ is deformable, $\phi_0\in\Aut_\ql(Z)$
and $\phi_0$ lifts to $\Phi'(0)\in N_{\GL(V)}(G)$.
  \end{proposition}
 
 \begin{proof} By \cite[Lemma 19]{Kuttler} $\Phi$ sends closed orbits to closed orbits. Hence $\Phi(0)=0$. Set  $\Phi_t(v) =t\inv\Phi(tv)$ for $t\in \C^*$, $v\in V$. Then $\Phi_t$ covers $\phi_t$ where $\phi_t(z)=t\inv\cdot\phi(t\cdot z)$, $z\in Z$. 
Let $z=\pi(v)\in Z$. Then
$$
\phi_0(z):=\lim_{t\to 0}t\inv\cdot\phi(t\cdot z)=\pi(\lim_{t\to 0}t\inv\Phi(tv))=\pi(\Phi'(0)(v))
$$
Hence $\phi$ is deformable to $\phi_0$ and $\phi_0$ lifts to $\Phi'(0)$. Since $\phi\inv$ is also deformable, we see that  $\phi_0\in\Aut_\ql(Z)$. Since $\phi$ and $\phi\inv$ are stratification preserving, so is $\phi_0$. By Proposition \ref{prop:normalizer}, $\Phi'(0)\in N_{\GL(V)}(G)$.
 \end{proof}

  \section{Lifting holomorphic isotopies}\label{sec:liftingcomplex}

Now we have to show   lifting in case $\phi$ is deformable and $\phi_0$ lifts to a $\sigma$-equivariant element of $\GL(V)$. We use results of Heinzner \cite{Heinzner91} and Heinzner and Kutzschebauch \cite{HKOka} to obtain a     lift of $\phi_0\inv\phi$ to $\HAut(V)^G$. It follows that   there is a biholomorphic $\sigma$-equivariant lift of $\phi$.

Let $X$ and $Y$ be complex analytic $G$-varieties. Let $K$ be a maximal compact subgroup of $G$.  Let
$\HH(X,Y)$ denote the holomorphic maps from $X$ to $Y$  and let  $\HH(X,Y)^G$ denote the $G$-invariant
ones. Similarly, we have $\HH(X,Y)^K$. Since  $K$ is Zariski dense in $G$  we have the following.

\begin{lemma}\label{k-invars-g-invars} $\HH(X,Y)^G=\HH(X,Y)^K$.
\end{lemma}

Let $U$ be a $K$-stable open subset of
$X$. We say that $U$ is {\em orbit convex\/} if for every $Z\in i\liek$ and $x\in U$, if
$\exp(Z)x\in U$, then $\exp(tZ)x\in U$ for $0\leq t\leq 1$. 

\begin{proposition}[(\cite{Heinzner91})] \label{prop:Kisenough}Let $U\subset V$ be $K$-invariant and orbit convex  where $V$ is a
$G$-module. Let $Y$ be a complex analytic $G$-variety. Then restriction gives an isomorphism
of $\HH(GU,Y)^G$ with $\HH(U,Y)^K$.
\end{proposition}

 Let $V$ be a $G$-module. Choose a $K$-invariant hermitian inner product on $V$. Then the \emph{Kempf-Ness set $\M$\/} of $V$ consists of the points $v\in V$ such that $T_v(Gv)$ is perpendicular to $v$. This set possesses some remarkable properties (see \cite{SchTopAlgQuots}). In particular,  $\M\to Z$ is surjective, proper and the fibers are $K$-orbits. Every closed orbit $Gv$ has a closest point to the origin which is necessarily in $\M$.

\begin{lemma}[(\cite{HKOka})]\label{lem:orbit-convex-exist} Let $V$ be a $G$-module. Then any $K$-stable subset of $\M$ has a neighborhood basis of orbit convex  subsets  $U$ such that $GU$ is $G$-saturated.
\end{lemma}

The following result is a holomorphic analogue of the isotopy lifting theorem of \cite{SchLifting}.

 \begin{theorem}\label{thm:isolifting} Assume that $V$ is admissible, $3$-principal or $2$-principal and orthogonal.  Assume that  we have a holomorphic family $\phi_t$ of elements of $\HAut(Z)$, $t\in\C$, where $\phi_0$ is the identity. Then there are $\Phi_t\in\HAut(V)^G$, $t\in [0,1]$, such that $\Phi_t$ lifts $\phi_t$ for each $t\in [0,1]$.
\end{theorem}

\begin{proof}
There is a holomorphic (complex) time dependent vector field $A(t,z)$ such that the integral of $A$ from $0$ to $t$ gives back $\phi_t$.   By \cite[Theorem 1.3, Remark 2.4]{SchVectorFields}, there is a holomorphic $G$-invariant vector field $B$ on $\C\times V$ which covers $A$. We show that  the integral $\Phi_t$ of $B$ exists for $0\leq t\leq 1$, hence  $\Phi_t$  lifts $\phi_t$, $t\in[0,1]$. 

Let $v_0\in V$ such that $Gv_0$ is closed. We may assume that $v_0\in\M$. Let $B_a$ denote $(t,v)\mapsto B(a+t,v)$, $a\in[0,1]$. Let $L$ denote the inverse image in $\M$ of $\phi_t(\pi(v_0))$, $t\in[0,1]$.  Then $L$ is compact  and we can integrate $B_a$ on $L$ for $t\in[0, 2\epsilon)$ for some $\epsilon>0$ and any $a\in[0,1]$. Choose an orbit convex neighborhood $U$ of $L$, where $GU$ is $G$-saturated, such  that we can integrate any $B_a$ along $[0,\epsilon]$ for $v\in U$. The flow  $\Phi^a_t$ of $B_a$ on $[0,\epsilon]\times U$ is $K$-invariant and holomorphic for any fixed $t$. Then $\Phi^a_t$ extends to a $G$-invariant flow on $[0,\epsilon]\times GU$  by Proposition \ref{prop:Kisenough}  and the extension is clearly the flow of $B_a$ on $[0,\epsilon]\times GU$.   Let $1/n<\epsilon$.
  Then $v_t:=\Phi^0_t(v_0)$ covers $\phi_t(\pi(v_0))$, $t\in[0,1/n]$, and it lies in $GU$ since $GU$ is $G$-saturated.
 Thus we can apply $\Phi^{1/n}_t$ to $v_{1/n}$ for $t\in[0,1/n]$ and we end up with $v_{2/n}\in GU$ which lies above $\phi_{2/n}(\pi(v_0))$. Thus the flow $\Phi_t$ of $B$ exists in a neighborhood of $v_0$ for $t\in[0,2/n]$. Continuing inductively we see that the flow $\Phi_t$ exists in a neighborhood of $v_0$ for $t\in[0,1]$.  But $v_0$ is arbitrary in $\M$ so that $\Phi_t$, $t\in[0,1]$, exists in an orbit convex neighborhood $U$ of $\M$ such that $GU$ is $G$-saturated. Then $\Phi_t$ extends to $GU=V$ and $\Phi_t$ is a lift of $\phi_t$ to $\HAut(V)^G$.
 \end{proof}

\begin{proof}[Proof of Theorem  \ref{thm:intro:lifting}] Use Theorem    \ref{thm:isolifting} and  Propositions \ref{prop:normalizer},   \ref{prop:autql0} and \ref{prop:liftsimpliesdiff}.
\end{proof}

 \section{Some examples}\label{sec:examples}
 
\begin{lemma}\label{lem:selfdual}
Let $V$ be an   $\SL_n$-module, $n\geq 3$, and let $\Phi\in\GL(V)$ such that $\Phi$ normalizes the image $H$ of $\SL_n$ in $\GL(V)$. If $V$ is not self-dual, then $\Phi$ induces an inner automorphism of $H$.
\end{lemma}

\begin{proof}
We may assume that $\Phi$ induces a diagram automorphism. If the automorphism is outer, then it sends every highest weight vector of $V$ to a highest weight vector of  $V^*$. Since $V$ is not self-dual,  the highest weights of $V^*$ are not all weights of  $V$. Hence $\Phi$ must induce an inner automorphism of $H$.
\end{proof}

Here are the  examples which arise in Theorem \ref{thm:citlifting} and  Remark \ref{rem:nolift}.

\begin{example} \label{ex:g2} (See \cite{SchCIT} and \cite[\S 5]{SchG2}.) Let $n\geq 4$. Then 
 $(V,G)=(n\C^7,\GG_2)$   is admissible. Since $V\simeq (\C^n)^*\otimes\C^7$ has a $\GL_n$-action commuting with the $\GG_2$-action,  $\GL_n$ acts on the invariants. We have generating  invariants $\alpha_{ij}$, $\beta_{ijk}$ and $\gamma_{ijkl}$ whose spans transform by the $\GL_n$-representations $\Sym^2(\C^n)$, $\wedge^3(\C^n)$ and $\wedge^4(\C^n)$, respectively. We use the notation $\psi_q^p$ or $\psi_q^p(\C^n)$ to denote the Cartan component of $\Sym^p(\wedge^q(\C^n))$.  

There are three ways to get a highest weight element of   $\psi^2_4(\C^n)$.  There is  $(\gamma_{1234})^2$, the determinant $\det(\alpha_{ij})_{i,j=1}^4$ and 
$$
\sum_{i,j=1}^4(-1)^{i+j}\alpha_{ij}\hat\beta_i\hat\beta_j.
$$
Here $\hat\beta_i=\beta_{klm}$ where $\{i,k,l,m\}=\{1,2,3,4\}$ and $k<l<m$. We will denote these elements as $\lambda(\psi_4^2(\gamma^2))$, $\lambda(\psi_4^2(\alpha^4))$ and $\lambda(\psi_4^2(\alpha\beta^2))$, respectively. The notation is meant to denote  a highest weight element of $\psi_4^2$ which is of the indicated degrees in the invariants.

From    \cite[\S 5]{SchG2}  there is a relation
\begin{equation}\label{eq:g2:eq1}
\lambda(\psi_4^2(\gamma^2))+\lambda(\psi_4^2(\alpha\beta^2))-\lambda(\psi_4^2(\alpha^4))=0.
\end{equation}
Since this is a relation of highest weight vectors, we actually have a whole $\GL_n$-representation of relations. Now assume that $n=4$. Then the  relations are generated by the one above (we have a hypersurface). Let $\phi$ act on $\O(Z)$ such that it fixes  the $\alpha_{ij}$  and $\beta_{ijk}$ and sends $\gamma_{1234}$ to its negative. If $\phi$ lifts to a linear $\Phi$, then we can assume that $\Phi$ centralizes $\GG_2$ since $\GG_2$ has no outer automorphisms. Hence $\Phi\in\GL(V)^G=\GL_4$. But $\Phi$ acts trivially on $\Sym^2(\C^4)$, hence it is plus or minus the identity. Then $\Phi$ fixes $\gamma_{1234}$, showing that $\phi$ does not lift.

Suppose that $n>4$. Then  we pick up more relations. For example, we have a relation
\begin{equation}\label{eq:g2:eq2}
\lambda(\psi_1\psi_5(\beta^2))+\lambda(\psi_1\psi_5(\alpha\gamma))=0.
\end{equation}
 Let $\phi\in\Aut_\ql(Z)$. The action of $\phi$ on $\Sym^2(\C^n)$ has to normalize the action of $\SL_n$ since $\SL_n$ times the scalars is   $\Aut_\ql(Z)^0$. By Lemma \ref{lem:selfdual}, $\phi$ induces an inner automorphism of $\SL_n$. Thus, modulo the image of an element of $\SL_n$ acting on $n\C^7$, we can assume that $\phi$ centralizes the action of $\SL_n$, hence $\phi$ acts as a scalar on each of the types of invariants. Changing $\phi$ by the image of a scalar in $\GL_n$,  we   can assume that $\phi$ acts trivially on $\Sym^2(\C^n)$.  The  relation \eqref{eq:g2:eq1} then shows that  $\phi$ fixes each $\beta_{ijk}$ or sends each $\beta_{ijk}$ to its negative. 
 In the latter case  we can cover this  by choosing $-I\in\GL_n$, so we may assume that $\phi$ fixes the $\beta_{ijk}$. Then \eqref{eq:g2:eq2} shows that $\phi$ acts trivially on the $\gamma_{ijkl}$. Hence we have a lift of $\phi$.
\end{example}

\begin{example} \label{ex:b3} (See \cite[\S 5]{SchG2}.)
Let $(V,G)=(n\C^8,\BB_3)$. Then $V$ is admissible for $n\geq 5$.  The invariants are generated by inner products $\delta_{ij}$ and skew symmetric functions $\epsilon_{ijkl}$ where $\epsilon_{1234}\in\wedge^4(\C^8)^{\BB_3}\subset \Sym^4(4\C^8)^{\BB_3}$. For $n=5$ we have a single generating relation
\begin{equation}\label{eq:b3:eq1}
\lambda(\psi_5^2(\delta^5))=\lambda(\psi_5^2(\delta\epsilon^2)).
\end{equation}
As in the $\GG_2$ case, the automorphism of $Z$ which  fixes the $\delta_{ij}$ and sends each $\epsilon_{ijkl}$ to its negative does not lift.  Suppose that  $n\geq 6$. Then we pick up the relation
\begin{equation}
\lambda(\psi_2\psi_6(\epsilon^2))+\lambda(\psi_2\psi_6(\delta^2\epsilon))=0.
\end{equation}
As in the $\GG_2$ case, the relations allow  one to show that any $\phi$ lifts.
\end{example}

\begin{example}
Let $(V,G)=(2n\C^n,\SL_n)$, $n\geq 2$.  Then $V$ is admissible. The generators are determinants of $n$ vectors and transform by the representation $W:=\wedge^n(\C^{2n})$ of $\GL_{2n}$. There is a natural $\SL_{2n}$-invariant bilinear form $\langle\, ,\rangle$ on $W$ which sends $\alpha$, $\beta$ to $\alpha\wedge\beta\in\wedge^{2n}(\C^{2n})\simeq\C$.  
The bilinear form corresponds to an isomorphism $\psi$ of $W$ with $W^*$ where $\psi(\alpha)(\beta)=\langle \alpha,\beta\rangle$. For the standard basis $e_i$ for $\C^{2n}$   let $\tau\colon \C^{2n}\to (\C^{2n})^*$ send $e_i$ to $e_i^*$ for each $i$. Recalling that $g\in\SL_{2n}$ acts on $(\C^{2n})^*$ by inverse transpose, which we denote by $\sigma(g)$, we have that $\tau(gv)=\sigma(g)\tau(v)$ for $v\in\C^{2n}$.  
  We have an induced mapping (also called $\tau$) from $W=\wedge^n(\C^{2n})$ to $W^*=\wedge^n((\C^{2n})^*)$ and composing with $\psi\inv$ we have an isomorphism $\phi$ of $W$ which is $\sigma$-equivariant. It is easy to see that $\phi$ preserves $Z$ which is    the  set of decomposable vectors of $W$. If $\phi$ lifts, then it lifts to   $\Phi\in\GL((\C^{2n})^*\otimes\C^n)$ where $\Phi$ normalizes $\SL_n$. By Lemma \ref{lem:selfdual}, we may assume that $\Phi$ centralizes $\SL_n$, hence $\Phi\in\GL_{2n}$. Thus conjugation by $\Phi$ induces an inner automorphism of  $\SL_{2n}$ and $\Phi$ cannot cover $\phi$ since conjugation by $\phi$ induces $\sigma$.  
  \end{example}

\begin{example}
Let $(V,G)=(k\C^n,\SO_n)$, $k\geq n\geq 3$. Then $V$ is admissible. We have inner product invariants $\alpha_{ij}$ whose span transforms as $\Sym^2(\C^k)$ under the  action of $\GL_k$ and determinant invariants $\beta_{i_1,\dots,i_n}$ whose span transforms as $\wedge^n(\C^k)$. Let $\phi$ be quasilinear. As before, we can assume that the action of $\phi$ centralizes that of $\SL_k$ so that $\phi$ acts on the invariants $\alpha_{ij}$ and $\beta_{i_1,\dots,i_n}$ by scalars. We may reduce to the case that $\phi$ acts as the identity on $S^2(\C^k)$. We have the  highest weight relation 
\begin{equation}\label{eq:son:eq1}
\lambda(\psi_n^2(\alpha^n))-\lambda(\psi_n^2(\beta^2))=0.
\end{equation}
Now \eqref{eq:son:eq1}
 implies that $\phi$ acts on the $\beta_{i_1,\dots,i_n}$ as    $\pm 1$. Consider an element of $\Orth_n\setminus\SO_n$. It fixes the $\alpha_{ij}$  and sends the $\beta_{i_1,\dots,i_n}$  to their negatives. Hence $\phi$ lifts.
\end{example}
\begin{example}
Let $(V,G)=(k\C^n,\Orth_n)$ where $k>n$. Then $V$ is admissible. We only have generators $\alpha_{ij}$ and arguing as above one sees that $\Aut_\ql(Z)$ is the image of $\GL_k$.
\end{example}

\begin{example}
Let $(V,G)=(k\C^{2n},\Sp_{2n})$, $k\geq 2n+2$ and $n\geq 2$. Then $V$ is admissible. The generators $\alpha_{ij}$ are skew in $i$ and $j$ and transform by the representation  $\wedge^2(\C^k)$ of $\GL_k$. Since $k\geq 6$ one can show as above that $\Aut_\ql(Z)$ is the image of $\GL_k$. 
\end{example}

\begin{example}
Let $(V,G)=(k\C^n+\ell(\C^n)^*,\SL_n)$ where $k+\ell\geq 2n$ and $k+\ell>2n$ if $k\ell=0$. Then $V$ is admissible. We leave the cases where $k\ell=0$ or $k<n$ or $\ell<n$ to the reader and now consider the case where $k$, $\ell\geq n$. We use symbols $\psi$ and ${\overline \psi}$ to differentiate between the representations of $\GL_k$ and $\GL_\ell$. Now $\GL_k\times\GL_\ell$ acts on the quotient $Z\subset\C^k\otimes\C^\ell\oplus \wedge^n(\C^k)\oplus\wedge^n(\C^\ell)$ where $\O(V)^G$ has corresponding generators $\alpha_{ij}$ (contractions), $\beta_{i_1,\dots,i_n}$ (determinants of elements of $k\C^n$) and $\gamma_{j_1,\dots,j_n}$ (determinants of elements of $\ell(\C^n)^*$). If $k\neq\ell$ then by Lemma \ref{lem:selfdual}  any $\phi\in\Aut_\ql(Z)$ induces an inner automorphism of  $\SL_k\times\SL_\ell$.  If $k=\ell$ we have a $\phi$ which interchanges the copies of $\SL_k=\SL_\ell$. But this automorphism obviously lifts. Thus we may assume that $\phi$ induces an inner automorphism and changing $\phi$ by an element of $\SL_k\times\SL_\ell$ we may assume that $\phi$ commutes with $\GL_k\times\GL_\ell$. Then $\phi$ acts as a scalar on our  three representations spaces. We have a relation 
\begin{equation}\label{eqsl:eq1}
\lambda(\psi_n\overline{\psi}_n(\beta\gamma))-\lambda(\psi_n\overline{\psi}_n(\alpha^n))=0.
\end{equation}
This shows that the scalars $a$, $b$ and $c$ on the $\alpha_{ij}$, etc.\ satisfy the relation $a^n=bc$. Let $b'$ and $c'$ be $n$th roots of $b$ and $c$. Let $b'$ (resp.\ $c'$) act by multiplication on $\C^k$ (resp.\ $\C^\ell$). Then we get the desired action on the invariants of type $\beta$ and $\gamma$ and we act by $a:=b'c'$ on the $\alpha_{ij}$ where $a^n=bc$. But $b'c'$ is an arbitrary $n$th root of $bc$. Hence $\phi$ lifts.
\end{example}

\begin{example}
Let $(V,G)=(k\C^n+\ell(\C^n)^*,\GL_n)$, $k$, $\ell>n$. Then the same ideas as for $\SL_n$ show that any $\phi\in\Aut_\ql(Z)$ lifts.
\end{example}

We have handled all the cases considered in Theorem \ref{thm:citlifting} and Remark \ref{rem:nolift}. Without burdening the reader with any more calculations we state the following proposition without proof. We denote by $R_j$ the   $\SL_2$-module  $S^j(\C^2)$.   

\begin{proposition}
Let $V$ be a $G$-module, $V^G=0$,  where $G=\SL_2$. We assume that $V\neq 4R_1$.
\begin{enumerate}
\item Suppose that  $Z$ is a hypersurface singularity. Then any $\phi\in\Aut_\ql(Z)$ lifts. 
\item Suppose that $V=2R_3$ (the unique case where $Z$ is a complete intersection defined by 2 equations). Then any $\phi\in\Aut_\ql(Z)$ lifts.
\end{enumerate}
 \end{proposition}
 
 \begin{problem}
 Does one always have lifting for  irreducible  admissible representations of $\SL_2$?
 \end{problem}
 
 \section{Tori}\label{sec:torus}

 Let $V$ be a $2$-principal $G$-module. Let $Z$ and $Z_\pr$ be as usual and let $Y$ denote $\quot V{G^0}$. Then $Z=Y/H$ where $H:=G/G^0$. Let $Y_\pr$ denote the inverse image of $Z_\pr$ in $Y$.   
  
 \begin{lemma}\label{lem:reduceconn}
Let $Y$, etc.\ be as above.  Then any $\phi\in\Aut(Z)$ has a lift $\Phi\in\Aut(Y)$.  
\end{lemma}  

 \begin{proof}
Since $Y_\pr\to Z_\pr$ is a finite cover, there are local holomorphic  lifts    of $\phi$ over small open subsets of $Z_\pr$. Since $Y\setminus Y_\pr$ has real codimension four in $Y$, $Y_\pr$ is simply connected and we can patch together the local lifts to give a global holomorphic lift $\Phi_\pr$ on $Y_\pr$. By \cite[Lemma 5.1.1]{PopovMichor}, $\Phi_\pr$ is algebraic, and $\Phi_\pr$ extends to a morphism $\Phi\colon Y\to Y$ covering $\phi$.  Similarly, there is a morphism $\Psi$ covering $\phi\inv$. Then $\Phi\circ\Psi$ and $\Psi\circ\Phi$ have to be   deck transformations, i.e.,  elements of $H$. Thus $\Phi\in\Aut(Y)$.
 \end{proof}
  
Let $G_s$  be the semisimple part of $G^0$ and let $S$ be the connected center of $G^0$.

 \begin{proposition}\label{prop:reducess}
Let $V$ be $2$-principal and let $\phi\in\Aut(Z)$. Then $\phi$ lifts to an automorphism $\Phi\in\Aut(Y)$ where $Y=\quot V{G_s}$.  
\end{proposition}

 \begin{proof}   Lemma \ref{lem:reduceconn} allows us to reduce   to the case that $G$ is connected. We may assume that $S$ is nontrivial. Possibly dividing out by a finite kernel   we may assume that $S$ acts effectively on $Y$, in which case the principal isotropy group is trivial. Let $Z_0:=Z\setminus Z_\pr$,   let $Y_0$ denote the inverse image of $Z_0$ in $Y$ and set $Y_\pr:=Y\setminus Y_0$. Then $\pi\colon Y_\pr\to Z_\pr$ is a principal $S$-bundle and $Y_0$ and $Z_0$ have complex codimension two in $Y$ and $Z$, respectively. Let $D$ be an irreducible hypersurface in $Y_\pr$ and let $\overline{D}$ be its closure in $Y$. Since $Y$ is factorial, $\overline{D}$ is the zero set of an irreducible element of $\O(Y)$. Hence the divisor class group of $Y_\pr$ is trivial which implies that $\HHH^1(Y_\pr,\O^*)$ is trivial.
 
Let $L\in\HHH^1(Z_\pr,\O^*)$ be a line bundle. Then the pull back $\pi^*L\in\HHH^1(Y_\pr,\O^*)$ is trivial and has an action of $S$.    The $S$-action on $Y_\pr\times\C$ sends $(y,\lambda)$ to $(s\inv y,c(y,s)\lambda)$, $s\in S$, $y\in Y_\pr$, $\lambda\in\C$ where $c(y,s)$ is a morphism  $Y_\pr\times S\to\C^*$ which sends $Y_\pr\times\{e\}$ to $1\in\C^*$. From \cite[Theorem 3]{Rosenlicht} and  \cite[Proposition 16.3]{HumphLinAlg} we see that $c(y,s)$ is a character $\chi(s)$ of $S$ which    is independent of $y\in Y_\pr$.   Let $\C_\chi$ denote a copy of $\C$ where $S$ acts via $\chi$. Then $L\simeq L_\chi:=Y_\pr\times^S\C_\chi$ and $\HHH^1(Z_\pr,\O^*)\simeq X(S)$, the character group of $S$.

Now $\phi^*L_\chi$ is isomorphic to a line bundle $L_{\tau(\chi)}$ for some $\tau(\chi)\in X(S)$. Using additive notation for the group structure of $X(S)$ 	 we see that $\tau(\chi_1+\chi_2)=\tau (\chi_1)+\tau(\chi_2)$ for $\chi_1$, $\chi_2\in X(S)$. Thus $\phi$ acts on $\HHH^1(Z_\pr,\O^*)\simeq X(S)$ by $\tau\in\GL_n(\Z)$ where $n$ is the rank of $S$. Let $\chi_1,\dots,\chi_n$ be a basis of $X(S)$.
For $j=1,\dots,n$ we have an isomorphism $\psi_j$ of $\phi^*L_{\chi_j}$ with $L_{\tau(\chi_j)}$. From $\psi_j$ we canonically obtain an isomorphism, denoted $\psi_j^{-1}$, of $\phi^*L_{-\chi_j}=(\phi^*L_{\chi_j})\inv$ with $L_{-\tau(\chi_j)}$.
Let $\chi=\sum n_j\chi_j$. Then define 
$$
\psi_\chi=\bigotimes_{j=1}^n \psi_j^{n_j}\colon (\phi^*L_\chi=\bigotimes_{j=1}^n (\phi^*L_{\chi_j})^{n_j})\to (L_{\tau(\chi)}=\bigotimes_{j=1}^n (L_{\tau(\chi_j)})^{n_j}).
$$
Then the isomorphisms $\psi_\chi$  give us an algebra isomorphism 
$$
\psi\colon \bigoplus_{\chi\in X(S)}\Gamma(Z_\pr,\phi^*L_\chi) \simeq \bigoplus_{\chi\in X(S)}\Gamma(Z_\pr,L_{\tau(\chi)})=\bigoplus_{\chi\in X(S)}\Gamma(Z_\pr,L_\chi).
$$

Let $\O(Y)_\chi$ denote the covariants of type $\chi$, that is, $f\in\O(Y)_\chi$ if $f(s\inv y)=\chi(s)f(y)$ for $y\in Y$ and $s\in S$. Then 
$$
\O(Y)=\bigoplus_{\chi\in X(S)}\O(Y)_\chi=\bigoplus_{\chi\in X(S)}\O(Y_\pr)_\chi=\bigoplus_{\chi\in X(S)}\Gamma(Z_\pr,L_\chi).
$$ 
Let $L_\chi(z)$ denote the fiber of $L_\chi$ at $z$. Then for $z\in Z_\pr$ we have
$$
\O(\pi\inv(z))\simeq\bigoplus_{\chi\in X(S)} L_\chi(z) \text{ and } \O(\pi\inv(\phi(z)))\simeq\bigoplus_{\chi\in X(S)} (\phi^*L_\chi)(z).
$$
Thus $\psi$ gives us a family of isomorphisms of the fibers $\pi\inv(z)$ and $\pi\inv(\phi(z))$   parameterized by $Z_\pr$. Hence we have a lift $\Phi_\pr$ of $\phi$ to $Y_\pr$ and $\Phi_\pr$ canonically extends to a lift  $\Phi$ of $\phi$ to $Y$. By construction, $\Phi^*$ sends $\O(Y)_\chi$ isomorphically onto $\O(Y)_{\tau(\chi)}$ for all $\chi\in X(S)$. Hence $\Phi\in\Aut(Y)$. 
\end{proof}
  
 \begin{proof}[Proof of Theorem \ref{thm:intro:liftingtori}] We are in the case that $G^0$ is a torus. Let $\phi\in\Aut(Z)$. We may assume that $\phi(\pi(0))=\pi(0)$. By  Proposition \ref{prop:reducess}, $\phi$ lifts to  $\Phi\in\Aut(V)$ and we also have a lift of $\phi\inv$. As in  Proposition \ref{prop:liftsimpliesdiff}, $\Phi(0)=0$ and $\phi$ is deformable to $\phi_0\in\Aut_\ql(Z)$ and $\phi_0$ lifts to $\Phi'(0)\in N_{\GL(V)}(G)$.
Set $\Phi_t=t\inv\circ\Phi\circ t$, $t\in \C^*$, and set $\Phi_0=\Phi'(0)$. Set $\Psi_t:=\Phi_t\circ\Phi'(0)\inv$.   
 Let $v\in V_\pr$. Then $\Psi_t(gv)=\gamma(t,v,g)\cdot \Psi_t(v)$ where $\gamma\colon\C\times V_\pr\times G\to G$ is a morphism. As usual, $\gamma$ extends to a morphism $\C\times V\times G\to G$. Since $\gamma(0,v,g)=g$ for all  $v$ and $g$,  $g\inv\gamma(t,v,g)$ always lies in $G^0$. For $g$ fixed, $(t,v)\mapsto g\inv\gamma(t,v,g)$ is  a morphism $\C\times V\to G^0$. Since the only units of $\O(\C\times V)$ are constant, $g\inv\gamma(t,v,g)$ is independent of $t$ and $v$.  It follows that $\gamma(t,v,g)=g$ for all $t$, $v$ and $g$. Hence all the $\Psi_t$ are $G$-equivariant and our original $\Phi$ is $\sigma$-equivariant.
  
 If we start out with $\phi\in\HAut(Z)$ we may assume that $\phi(\pi(0))=\pi(0)$ and then we know that $\phi$ is deformable to a $\phi_0\in\Aut_\ql(Z)$, which, by Proposition \ref{prop:reducess} and Proposition \ref{prop:normalizer}, lifts to a  $\sigma$-equivariant automorphism of $V$. Then by Theorem \ref{thm:intro:lifting}, $\phi$ has a $\sigma$-equivariant   lift to $\HAut(V)$.
 \end{proof}

Later we will need a local version of lifting.
 \begin{proposition}\label{prop:locallift}
 Assume that $V$ is a $2$-principal $G$-module where $G^0$ is a torus. Let $B$ and $B'$ be neighborhoods of $z_0:=\pi(0)$ and let $\phi\colon B\to B'$ be biholomorphic where $\phi(z_0)=z_0$.
 Then, after perhaps shrinking $B$ and $B'$, we can find a $\sigma$-equivariant biholomorphic lift $\Phi$ of $\phi$.
 \end{proposition}
 
 \begin{proof}   Define $\phi_t$ as usual. Then, since $V$ is admissible, the proof of Theorem \ref{thm:admissibleimpliesdiff}   shows that $\phi_0:=\lim_{t\to 0}\phi_t$ exists. Since all the $\phi_t$ preserve $z_0$, we can shrink $B$ so that $\phi_t(B)\subset B'$, $t\in[0,1]$. We know that  $\phi_0$ lifts to $N_{\GL(V)}(G)$, hence we can reduce to the case that $\phi_0$ is the identity. Then the argument of Theorem \ref{thm:isolifting} shows that we have an equivariant biholomorphic lift $\Phi_t$ of $\phi_t$, $0\leq t\leq 1$.   Then $\Phi_1$ is the required lift of $\phi$.
 \end{proof}

\section{The method of Kuttler}\label{sec:kuttler}

We  give a proof of Theorem \ref{thm:genkuttler}. Let $V$ be a $4$-principal $G$-module which we may assume is  faithful.  Let $V_0$ denote $V\setminus V_\pr$. By Proposition \ref{prop:reducess}, lifting for elements of $\Aut(Z)$ reduces  to establishing lifting for the action of $G_s$ where $G_s$ is the semisimple part of $G^0$. Since $(V,G)$ is $4$-principal, it is easy to see that $(V,G_s)$ is also $4$-principal. Hence we may assume that $G$ is connected semisimple. In Theorem \ref{thm:genkuttler} we assume that
\begin{equation}\label{assumption:lieg}\tag{*}
V=\lieg\oplus V', \text{ i.e.,  $V$ contains a copy of the adjoint representation of $G$.} 
\end{equation}
\begin{lemma}\label{lem:injective}
The mapping $G\to\GL(V')$ is injective.
\end{lemma}
\begin{proof}
Let $K$ be the kernel. If $\dim K>0$, then $K$ contains a simple component of $G$ whose representation on $\lieg$ has codimension one strata. Hence $V$ is not $2$-principal and we have a contradiction. Thus $K$ is finite and central. Since $K$ acts trivially on $\lieg$ we must have that $K=\{e\}$. 
\end{proof}

Let $\lieA(Z)$  denote the vector fields on $Z$ (derivations of $\O(Z)$) and define $\lieA(V)$ similarly.
Since the elements of $\lieg$ act on $V$ as vector fields, we have an inclusion $\tau\colon\lieg\to\O(V)\otimes V\simeq\lieA(V)$. Thus we have an injection (also called $\tau$) from $M_\lieg:=\O(V)\otimes\lieg$ to $M_V:=\O(V)\otimes V$. Let $M$ denote $M_V/M_\lieg$.

\begin{lemma}\label{lem:depth3}
The module $M$  has depth $3$ relative to the ideal of $V_0$.   
\end{lemma}

\begin{proof} The free $\O(V)$-modules $M_\lieg$ and $M_V$ have depth 4 with respect to the ideal of $V_0$, hence $M$ has depth at least 3.  
\end{proof} 

 Set $E:=(M_\lieg)^G$ and $F:=(M_V)^G=\lieA(V)^G$. Then  $\tau(E)\subset F$. The image of $\tau$ consists of invariant vector fields tangent to the $G$-orbits.   We have a morphism $\pi_*\colon \lieA(V)^G\to\lieA(Z)$ which just restricts an element of $\lieA(V)^G$ to $\O(V)^G=\O(Z)$. By   \cite[Theorem 8.9]{SchLiftingDOs} we have that 
\begin{equation}\label{assumption:exact}\tag{**}
0\to E\xrightarrow\tau F\xrightarrow{\pi_*}\lieA(Z)\to 0 \text{ is exact.}
\end{equation}

  Let $\E$, $\F$ and $\G$ be the sheaves of $\O_Z$-modules corresponding to $E$, $F$ and $\lieA(Z)$, respectively. Let $\F'$ be the sheaf corresponding to $(\O(V)\otimes V')^G$, so we have that $\F\simeq\E\oplus\F'$. The sections of $\E$ over $Z_\pr$ are the sections of the vector bundle $V_\pr\times^G\lieg$  so that $\E$ is locally free on $Z_\pr$. Similarly, $\F$ is locally free on $Z_\pr$ and the quotient $\G$ is locally free as well.
  
  Let $L$ be a finite $\O(Z)$-module and let $\LL$ denote the corresponding sheaf on $Z$, which we assume to be locally free on $Z_\pr$.   Let $\pi^*\LL$ denote $(\pi|_{V_\pr})^*\LL$. Then
 $\pi^*\LL$ is locally free on $V_\pr$, hence reflexive, and $N:=\Gamma(V_\pr,\pi^*\LL)$ is a finitely generated reflexive $\O(V)$-module. (See \cite[Remark 13]{Kuttler}. The argument only needs that $\codim V\setminus V_\pr\geq 2$.)\  Let $\pi^\#\LL$ denote the  coherent sheaf of $\O_V$-modules corresponding to $N$. For any Zariski open subset $U$ of $V$, $\Gamma(U,\pi^\#\LL)=\Gamma(U\cap V_\pr,\pi^*\LL)$ so  that $\pi^\#\LL=i_*\pi^*\LL$ where $i\colon V_\pr\to V$ is inclusion.    
 
  Since $\pi|_{V_\pr}$ is flat we have a complex of coherent sheaves  
\begin{equation}\label{eq:exactpisharp}
0\to\pi^\#\E\to\pi^\#\F\to\pi^\#\G\to 0
\end{equation}
which is  exact on $V_\pr$.  
Since $V_\pr\times^G\lieg$ is locally trivial over $Z_\pr$, $\pi^*\E$ is locally the sheaf of sections of $V_\pr\times\lieg$. Thus
the global sections of $\pi^\#\E$ are a  reflexive module over $\O(V_\pr)=\O(V)$ which agrees with $M_\lieg$ on $V_\pr$, hence the global sections are $M_\lieg$. Similarly, the global sections of $\pi^\#\F$ are $M_V$.  The global sections of $\pi^\#\G$ are a reflexive module which agrees with the locally free module $M$ over $V_\pr$. Hence the global sections are the double dual $M^{**}$. Since $M$ has depth 3 (2 will do) with respect to the ideal of $V_0$ we have that $M=M^{**}$ (\cite[Lemma 8.3]{SchLiftingDOs}). Hence $\pi^\#\G$ corresponds to $M$ and \eqref{eq:exactpisharp} is exact on $V$.
  
Let $\phi\in\Aut(Z)$. Then applying pull back by $\phi$ we obtain an exact sequence  
\begin{equation}\label{eq:exactphi}
0\to \phi^*\E\to \phi^*\F\xrightarrow{\pi_*} \phi^*\G\to 0.
\end{equation}
We alter the map $\pi_*$ in \eqref{eq:exactphi} by composing it with    $d\phi\inv\colon T_{\phi(z)}Z\to T_z(Z)$, $z\in Z$. On vector fields composing with $d\phi\inv$  is the automorphism   $A\mapsto (\phi\inv)_*A:=\phi^*\circ A\circ (\phi\inv)^*$, $A\in\lieA(Z)$, and $(\phi\inv)_*$ gives an isomorphism of $\phi^*\lieA(Z)$ with $\lieA(Z)$.   Thus we have a complex 
\begin{equation}\label{eqn:exactsheaf}
0\to\pi^\#\phi^*\E\to\pi^\#\phi^*\F\to\pi^\#\G\to 0 
\end{equation}
which is exact on $V_\pr$. The key idea, following  \cite{Kuttler}, is to show that $\Gamma(V_\pr,\pi^\#\phi^*\F')=\Gamma(V,\pi^\#\phi^*\F')$  is a free $\O(V)$-module.  

 Now we take cohomology over $V_\pr$ to get an exact sequence
 $$
0\to\Gamma(V,\pi^\#\phi^*\E)\to\Gamma(V,\pi^\#\phi^*\F)\to M\to\HHH^1(V_\pr,\pi^\#\phi^*\E)\to\HHH^1(V_\pr,\pi^\#\phi^*\F)\to \HHH^1(V_\pr,\pi^\#\G)
 $$
 where the last module is zero since $M$ has depth 3 with respect to the ideal of $V_0$.  
The map $\HHH^1(V_\pr,\pi^\#\phi^*\E)\to\HHH^1(V_\pr,\pi^\#\phi^*\F)$ is a surjective morphism of finitely generated $\O(V)$-modules  (\cite[Remark 13]{Kuttler}). Since $\pi^\#\phi^*\F$ contains $\pi^\#\phi^*\E$ as a direct summand,  localizing at   points of $V$ and applying Nakayama's lemma we see that $\HHH^1(V_\pr,\pi^\#\phi^*\E)\to\HHH^1(V_\pr,\pi^\#\phi^*\F)$ has to be injective. It follows that $\Gamma(V,\pi^\#\phi^*\F)$ maps onto $M$. Thus we have exact sequences of finite $\O(V)$-modules
 
 \begin{equation}\label{eq:firstglobal}
 0\to M_\lieg\to M_V\to M\to 0 \text{ and}
 \end{equation}
  \begin{equation}\label{eq:secondglobal}
 0\to\Gamma(V,\pi^\#\phi^*\E)\to\Gamma(V,\pi^\#\phi^*\F)\to M\to 0.
 \end{equation}
 Let 
$$
\Psi\colon M_V\oplus\Gamma(V,\pi^\#\phi^*\F)\to M\to 0 
$$
be the difference of the two surjections.

\begin{lemma}
We have an exact sequence 
\begin{equation}\label{eq:first}
0\to\Gamma(V,\pi^\#\phi^*\E)\to\Ker\Psi\to M_V\to 0.
\end{equation}
\end{lemma}

\begin{proof}
In \eqref{eq:first} we have 
$$
\Gamma(V,\pi^\#\phi^*\E)\to \Gamma(V,\pi^\#\phi^*\F)\subset  M_V\oplus\Gamma(V,\pi^\#\phi^*\F).
$$
Then exactness of  \eqref{eq:secondglobal} implies that  $\Ker\Psi/\Gamma(V,\pi^\#\phi^*\E)$ maps isomorphically onto $M_V$ via projection. 
\end{proof}

Since $M_V$ is a projective $G$ and $\O(V)$-module, we can find a $G$-equivariant $\O(V)$-module splitting of \eqref{eq:first}. Applying the same reasoning to $\phi\inv$ we have an   exact sequence
$$
0\to\Gamma(V,\pi^\#(\phi\inv)^*\E)\to\Gamma(V,\pi^\#(\phi\inv)^*\F)\to M\to 0
$$
and a split exact sequence
\begin{equation}\label{eq:second}
0\to\Gamma(V,\pi^\#(\phi\inv)^*\E)\to\Ker\widetilde\Psi\to M_V\to 0
\end{equation}
where 
$$
\widetilde\Psi\colon M_V\oplus\Gamma(V,\pi^\#(\phi\inv)^*\F)\to M\to 0
$$
is the difference of the two relevant  surjections.

For the rest of our argument   we only use that $\codim V\setminus V_\pr\geq 2$. Consider  a $G$ and  $\O(V)$-module $L$. Let $\LL$ be the  sheaf of $\O_Z$-modules corresponding to $L^G$. We assume that $\LL$ is locally free on $Z_\pr$. Then we have the reflexive module $\Gamma(V,\pi^\#\phi^*\LL)$ and $L\to \Gamma(V,\pi^\#\phi^*\LL)$  is a functor $\mathbb T$. If $L=\Gamma(V,\pi^\#(\phi\inv)^*\E)$, then $L^G\simeq\Gamma(Z,(\phi\inv)^*\E)$ and ${\mathbb T}(L)\simeq M_\lieg$, both isomorphisms being canonical. Similarly,  $\mathbb T$ sends $\Gamma(V,\pi^\#(\phi\inv)^*\F)$ to $M_V$. Thus applying $\mathbb T$ to \eqref{eq:second} we obtain a split exact sequence
\begin{equation}\label{eq:third}
0\to M_\lieg\to\Ker\Psi\to\Gamma(V,\pi^\#\phi^*\F)\to 0.
\end{equation}
From the split exact sequences \eqref{eq:first} and \eqref{eq:third} we see that
$$
\Ker\Psi\simeq M_\lieg\oplus M_{V'}\oplus\Gamma(V,\pi^\#\phi^*\E)\simeq M_\lieg\oplus\Gamma(V,\pi^\#\phi^*\E)\oplus\Gamma(V,\pi^\#\phi^*\F')
$$
where $M_{V'}=\O(V)\otimes V'$.
It follows that for any $v\in V$, the minimal number of generators of  the  $\O_{V,v}$-module $\Gamma(V,\pi^\#\phi^*\F')\otimes_{\O(V)}\O_{V,v}$  is $\dim V'$. Thus  
$\Gamma(V,\pi^\#\phi^*\F')$ is a projective $\O(V)$-module, hence free by Quillen and Suslin. 
 
   Now $P:=V_\pr\to Z_\pr$ is a principal $G$-bundle as is $\phi^*P$. The pull-back $P':=(\pi|_{V_\pr})^*\phi^*P$ is  a principal $G$-bundle over $V_\pr$. We have associated vector bundles $P\times^G V'$, $\phi^*P\times^G V'$ and $P'\times^G V'$. It is standard that the vector  bundle $P'\times^GV'$ is trivial if and only if the principal bundle $P'\times^G\GL(V')$ is trivial. The sections of $P\times^GV'$  and $\F'$ over $Z_\pr$ are the same, the sections of $\phi^*P\times^G V'$ and $\phi^*\F'$ over $Z_\pr$ are the same and the sections of $P'\times^GV'$ and $\pi^\#\phi^*\F'$  over $V_\pr$ are the same. We have shown  that $\Gamma(V,\pi^\#\phi^*F')$ is a free $\O(V)$-module, hence   $P'\times^GV'$ is a trivial vector bundle and $P'\times^G\GL(V')$ is a trivial principal bundle. Now $G\to\GL(V')$ is injective and it follows  that
$P'$ is the pull-back of the principal $G$-bundle $\GL(V')\to\GL(V')/G$ via a morphism $V_\pr\to \GL(V')/G$. Since $\GL(V')/G$ is affine the morphism extends to $V$. Hence $P'$ extends to a principal $G$-bundle $P''$ over $V$. By \cite{Raghunathan} $P''$ is trivial, hence so is $P'$. Now $P=V_\pr$, so that    $P'=P\times_{Z_\pr}\phi^*P\simeq P\times G$. Thus $P'$ has a section. This is the same thing as a mapping $\Phi$ of $P$ to $\phi^*P$ which respects  the fibers of the maps to $Z_\pr$.   Since $\phi^*P$ has fiber $\pi\inv(\phi(\pi(v)))$ at $\pi(v)$, $\Phi$ is a map of $V_\pr$ to $V_\pr$ which is a lift of $\phi$. Then $\Phi$ extends to a morphism of $V$ to $V$ which lies over $\phi$.  
     This completes the proof of Theorem \ref{thm:genkuttler}.

 \section{Multiples of the adjoint representation}\label{sec:adjoint}  
We give a proof of Theorem \ref{thm:adjoint}. Recall that $V=\oplus_{i=1}^d r_i\lieg_i$ where  $r_i\geq2$ unless $\lieg_i=\liesl_2$, in which case $r_i\geq 3$. Let $G_i$ denote the adjoint group of $\lieg_i$ and set $G=G_1\times\dots\times G_d$. Let $Z_i$ denote $\quot{(r_i\lieg_i)}{G_i}$. Then $Z=Z_1\times\dots\times Z_d$. We cannot just apply Theorem \ref{thm:genkuttler} because $\pi\inv(Z\setminus Z_\pr)$ can have codimension less than four. If $S$ is a stratum of $Z$ with
$\codim S\geq 6$, we show that $\codim\pi\inv(S)\geq 4$  so we have to worry about the strata of codimension at most $5$. We  classify them and show that we can reduce to the case that  they are fixed  by our automorphisim $\phi$ of $Z$. The corresponding isotropy groups are tori or finite, so we can use Proposition \ref{prop:locallift}  to find local holomorphic lifts of $\phi$ over these strata. This enables us to show that the sequence (\ref{eqn:exactsheaf}) is exact on $V^\dag$ where $V^\dag$ is the inverse image of the strata of $Z$ of codimension at most 5 (so $V^\dag$ is open). Since $\codim (V\setminus V^\dag)\geq 4$, the rest of the proof of Theorem \ref{thm:genkuttler} goes through.
  
The strata of $Z$ are products $S_1\times\dots\times S_d$ where $S_i$ is a stratum of  $Z_i$. We will see that, unless   $S_i$ is the principal stratum of $Z_i$, it  has codimension at least three. Thus the strata of $Z$ of codimension at most five are the product of an $S_i$ of codimension at most five with the principal strata of the other factors. 
   
   We begin with the case where $V=r\lieg$, $\lieg$ is simple, and $G$ is the adjoint group of $\lieg$ (i.e., $d=1$).  Let  $S=Z_{(H)}$ be a stratum with associated slice representation $(W,H)$.  We have an $H$-module decomposition $W=W^H\oplus W'$ where $W'$ is an $H$-module.   
   We have the equality of $H$-modules: $V=W\oplus\lieg/\lieh$. Since $V=r\lieg$ we get the formula
   $$
   W=(r-1)\lieg\oplus\lieh.
   $$  
   Note that $W$ is orthogonal. It follows from the slice theorem that the codimension of $S$ in $Z$ is $\dim\quot{W'}H$ and that the codimension of $\pi\inv(S)$ is the codimension of the null cone $\NN(W')$ in $W'$.

The case where $\lieg=\liesl_2$ is trivial and is summed up in the following lemma, which we leave to the reader. For $H=\C^*$ we denote by $\nu_j$ the one-dimensional module of weight $j\in\N$. We let $\theta$ denote a   trivial module of any dimension.
   
    \begin{lemma}\label{lem:sl2}
   Let $V=r\liesl_2$ and $G=\PSL_2$ where $r\geq 3$. Then there are two non principal strata and slice representations as follows.
   \begin{enumerate}
\item $(W,H)=((r-1)(\nu_1+\nu_{-1})+\theta,\C^*)$. We have $\codim S=2r-3$ and $\codim\pi\inv(S)=r-1$.
\item $(W,H)=(V,G) $ so $S=\{\pi(0)\}$. We have    $\codim S=3r-3$ and $\codim\pi\inv(S)=2r-1 $.
\end{enumerate}
   \end{lemma} 
      
 For the next few results we assume that $\rank\lieg\geq 2$.   Let $S$ be a stratum with slice representation $(W=W^H\oplus W',H)$. 
      If $H$ is finite, the codimension  of $\pi\inv(S)$ is $\dim W'$, else 
from   \cite[Corollary 10.2]{SchLifting} we have the estimate
 $$
\codim\pi\inv(S)= \codim\NN(W')\geq \frac 12(\dim\quot {W'}H+\rank H+\mu(W'))
 $$
 where $\mu(W')$ is the multiplicity of the zero weight for the maximal torus of $H^0$.

  \begin{lemma}\label{lem:codim}
Let $S$ be a stratum  with slice representation $(W,H)$. 
 \begin{enumerate}
\item If $\codim S\geq 6$, then $\codim\pi\inv(S)\geq 4$.
\item If $\codim S\leq 5$, then $H^0$ is a torus.
\end{enumerate}
   \end{lemma}
 
 \begin{proof}  For (1) note that $\codim \pi\inv(S)\geq 6$ if $H$ is finite, else $\rank H\geq 1$ and our estimate above gives that $\codim\pi\inv(S)\geq 7/2$, hence $\codim\pi\inv(S)\geq 4$. For (2), suppose that  $H$ has semisimple rank at least two. Then the action of $H$ on $W'$ contains at least twice   $[\lieh,\lieh]$ which implies that $\dim\quot{W'}H\geq 6$. If $H$ has semisimple rank 1, then $W'$, as $\SL_2$-module, contains more than $2\liesl_2$ since this representation  has   codimension one strata. 
Thus, at a minimum, we have an orthogonal representation of $\SL_2$ times a torus $T$  
of the form $2\liesl_2+R_{2m+1}$, $m\geq 1$, or of the form $2\liesl_2+R_{m}\otimes\nu_\alpha+R_m\otimes\nu_{-\alpha}$, $m\geq 1$. Recall that $R_m=S^m(\C^2)$ and we denote by $\nu_\alpha$  the one-dimensional representation where $T$ acts by the character $\alpha$. Thus we get a quotient of dimension at least 6. Hence $H^0$ is a torus and we have (2).
   \end{proof}

We call a stratum $S$ \emph{subprincipal\/} if  its slice representation $(W,H)$ is nontrivial and the  closed orbits of $(W,H)$  are either fixed points or principal. Then $S$ is not in the closure of any stratum except the principal one. From \cite{Kuttler} we borrow the following.

\begin{lemma} Let $V=r\lieg$ be as above and let $T$ denote a maximal torus of $G$.
 Let $S$ be a subprincipal stratum. Then $H$ is conjugate to a subgroup of $T$ which is finite or a one-dimensional torus.
\end{lemma}

\begin{corollary} \label{cor:finiteH}Suppose that $\codim S\leq 5$ and that $H$ is finite.  
\begin{enumerate}
\item  If $S$ is subprincipal, then it has codimension $4$.
\item  If $S$ is not subprincipal, it has codimension $5$ and is contained in the closure of  a subprincipal stratum $S'$ of codimension four  where the associated isotropy group $H'$ is finite.
\end{enumerate}
\end{corollary}

\begin{proof}
For (1) we know that $H$ is contained in a maximal torus $T$ of $G$. The action of $H$ on $(r-1)\lieg$ (which is the slice representation up to trivial factors) acts on the root spaces $\lieg_\alpha$ via the roots $\alpha$ applied to $H$ (thought of as characters of $T$). Since the roots appear in pairs $\pm\alpha$, the dimension of $W'$ is even, hence so is the codimension of $S$. If this codimension is two, then $r=2$ and all the simple roots have value one on $H$, except for a pair $\pm \alpha$. But then there is an adjacent simple root $\beta$ such that $\alpha+\beta$ is a root and $\alpha+\beta$ applied to $H$ is nontrivial. Hence we have a contradiction, which gives us (1).

For (2) we know that $S$ corresponds to  a non principal  slice representation of a finite group, hence there is a   stratum $S'$ of codimension less than 5 corresponding to a finite group.  By (1) $S'$ has codimension 4,  hence we have (2).
\end{proof}
   \begin{corollary}\label{cor:codim2strata}
  Let $V=r\lieg$ be as above. Then $Z$ has no codimension two strata. For any codimension three stratum we have    $(W,H)=(2\nu_1+2\nu_{-1}+\theta,\C^*)$.
      \end{corollary}
  
  \begin{proof}
  We know that $Z$ has no codimension one strata \cite[Proposition 3.1]{SchVectorFields}. If $S$ is of codimension two, it is subprincipal. By Corollary \ref{cor:finiteH} $H$ cannot be finite, so we have that $H=\C^*\subset T$. As in the proof above, there are at least two pairs of roots $\pm\alpha$, $\pm\beta$ which are non trivial on $H$. Hence $W$ has at least four nonzero weights and $\codim S\geq 3$. Thus there are no codimension two strata. If $S$ has codimension three, then it must be subprincipal, hence the isotropy group is $\C^*$. But a subprincipal orthogonal faithful representation of $\C^*$ has to have  weights in pairs $\pm 1$.  
\end{proof}

  \begin{lemma} \label{lem:codim3}
   Let $V=r\lieg$ be as above, but allow   $\rank G=1$ in which case $r\geq 3$. Suppose that there is a stratum $S$ of codimension three. Then $(r\lieg,G)=(2\liesl_3,\PSL_3)$ or $(3\liesl_2,\PSL_2)$.
    \end{lemma}
    
\begin{proof} Lemma \ref{lem:sl2} takes care of the case where $G=\PSL_2$, so assume that $\rank\lieg\geq 2$. Corresponding to $S$ we have  $H=\C^*$ which acts with nonnegative weights on the positive root spaces.  If there were two simple roots $\alpha$ and $\beta$ which are nontrivial on $H$, then either $\alpha+\beta$ is a root and is nontrivial  on $H$  or $\alpha+\gamma$ is nontrivial where $\gamma$ is a third simple root. Hence only one simple root $\alpha$ is nontrivial on $H$. Then the other root not vanishing on $H$ is of the form $\alpha+\beta$ where $\beta$ is adjacent to $\alpha$. We clearly must have $r=2$, else there are more than 4 nonzero weight spaces of $H$.
  Moreover, the root $\alpha$ must be adjacent to at most one other simple root   $\beta$ and $\alpha+\beta$ must be the only positive root of $\lieg$ for which $\alpha$ and $\beta$ have positive coefficients. Hence the highest root is $\alpha+\beta$, i.e., $\lieg=\liesl_3$. Finally, one only has to check that $2\liesl_3$ has the desired slice representation.
     \end{proof}

\begin{remark}\label{rem:codim3}
Examining the slice representations of $3\liesl_2$ and $2\liesl_3$ one sees the following where $S$ denotes the codimension three stratum.
\begin{enumerate}
\item If $V=3\liesl_2$, then $S$ contains only a stratum of codimension $6$ in its closure (the image of the fixed point).  
\item If $V=2\liesl_3$, then $S$ contains a stratum of codimension $4$ in its closure (the image of the points with isotropy group the maximal torus) and the next stratum has codimension $6$ (corresponding to an isotropy group of semisimple rank $1$). The closure of $S$ is all the non principal strata of $Z$.
\end{enumerate}
\end{remark}

\begin{corollary}\label{cor:sl3torus}
Suppose that   $\codim S=4$. Then either
\begin{enumerate}
\item $S$ is subprincipal and $H$ is finite, or
\item $H$ is the maximal torus of $\PSL_3$, $V=2\liesl_3$ and $S$ is  not subprincipal.
\end{enumerate}
\end{corollary}

\begin{proof}
Suppose that $S$ is not subprincipal.  
Then it is in the closure of a codimension 3 stratum and Remark \ref{rem:codim3} gives us the desired result. If $S$ is subprincipal and $H$ is not finite, then $H=\C^*$ and $W'$ has dimension 5, which is impossible. Hence $H$ is finite. 
\end{proof}

\begin{corollary}\label{lem:rankT}
Suppose that $S$ has codimension five. Then $H$ has rank at most one.
\end{corollary}

\begin{proof} 
By Lemma \ref{lem:codim}, $H^0$ is a torus.
Suppose that  $H$ has rank $k>1$. Let $\alpha$ be a character of the $H^0$-action on $W'$ with associated one-dimensional weight space $\nu_\alpha$. Then we have an $H^0$-stable subrepresentation $\nu_\alpha+\nu_{-\alpha}$ of $W'$ and there is a closed orbit in the subrepresentation with isotropy group of rank $k-1$. 
This corresponds to a stratum $S'$ of codimension 3 or 4.  By Lemma \ref{lem:codim3},   Remark \ref{rem:codim3} and Corollary \ref{cor:sl3torus} there are no cases with 
an $S$ and $S'$ with positive dimensional isotropy. Hence $H$ has rank at most one.
\end{proof}

 \begin{proposition}\label{prop:classifS} Suppose that $\codim S$ is $4$ or $5$ and that $H$ is finite. Then $\codim S=4$, $V=2\lie{so}_5$ and $H=\pm 1$ acting by scalar multiplication on $W'=\C^4$.
Suppose that $S$ has codimension $5$. Then $H=\C^*$ and we are in one of the following cases.
\begin{enumerate}
\item $V=4\liesl_2$. Then $S$ is subprincipal and the only  stratum  in its closure  has codimension $9$. 
\item $V=2\lie{so}_5$ and $W'=3\nu_1+3\nu_{-1}$. Then $S$ is subprincipal and the maximal stratum in its closure is that of the maximal torus and has codimension $6$.
\item $V=2\lie{so}_5$ and $W'=2\nu_1+2\nu_{-1}+\nu_2+\nu_{-2}$.  Then $S$ is not subprincipal, it lies in the closure of the stratum of codimension $4$ and  the maximal stratum in the closure of $S$ is that of the maximal torus and has codimension $6$.
\item $V=2\liesl_4$. Then $S$ is subprincipal and the maximal stratum in its closure   has codimension $8$.  
\end{enumerate}
\end{proposition}

\begin{proof} We first consider the case where $\codim S=5$ and $H$ has rank $1$.
We use some of the same techniques as \cite{Kuttler}. Let $v\in V$ such that $\pi(v)\in S$. Then there is a $1$-parameter subgroup $\lambda\colon\C^*\to G$ such that $v_0:=\lim_{t\to 0}\lambda(t)\cdot v$ exists where $Gv_0$ is closed and $G_{v_0}$ is conjugate to $H$. Applying an element of $G$ we may assume that $G_{v_0}=H$. Since $\lambda$ fixes $v_0$, $\lambda$ is a $1$-parameter subgroup of $H$ where $H^0=\C^*$. Let $\lie P_\lambda$ denote the parabolic subalgebra corresponding to $\lambda$. Thus $\lie P_\lambda=\{x\in \lieg\mid \lim_{t\to 0}\lambda(t)\cdot x$ exists$\}$ so that $v\in r\lie P_\lambda$, by definition. Hence $\pi\inv(S)\subset G\cdot r\lie P_\lambda$. Now $\lie P_\lambda\subset \lie P_\mu$ where $\lie P_\mu$ is a maximal parabolic corresponding to  a 1-parameter subgroup   $\mu$ of $G$. Note that any point $v$ of $r\lie P_\mu$ has the property that $\lim_{t\to 0}\mu(t)\cdot v\in V^\mu$.  

Now let $v\in V^\lambda$  be on a closed orbit with isotropy group $H$. Since $V^\lambda\subset r\lie P_\lambda\subset r\lie P_\mu$, $\mu(t)\cdot v$ has limit a point where the isotropy group contains $\mu(\C^*)$. But $Gv$ is closed, so that $\mu(\C^*)$ must be conjugate to $\lambda(\C^*)$. Hence we see that $\lie P_\lambda$ was already maximal.

Let $(W,H)$ be the slice representation of $S$ with nontrivial part $W'$.   We may assume that the positive roots pair positively with $\lambda$.  Let $\alpha_1,\dots,\alpha_\ell$ be the simple roots of $\lieg$. Then all the root spaces $\{\lieg_{-\alpha_i}\}_{i\in I}$ lie in $\lie P_\lambda$ where $I$ has cardinality $\ell-1$. 
The kernel of the  $\{\alpha_i\}_{i\in I}$ is $H^0\simeq \C^*$.

Now we calculate the slice representation of $H$, restricted to $H^0$. But, up to trivial factors,  this is just $(r-1)\lieg$, considered as a representation of $H^0$. Thus this representation has at most 6 nonzero weights. But it is easy, by inspecting root systems, to show the following.
\begin{enumerate}
\item Let $\lieg=\AA_4$ and consider the roots restricted to a candidate for $\lambda$ (i.e., three of the simple roots are trivial on $\lambda$). Then the action of $\lambda$ on $\lieg$ has at least four pairs of nonzero weights.
\item Let $\lieg=\BB_3$, $\CC_3$, $\BB_4$ or $\CC_4$. Then the action of $\lambda$ on $\lieg$ has at least five pairs of nonzero weights.
\item Let $\lieg=\DD_4$. Then the action of $\lambda$ on $\lieg$ has at least six pairs of nonzero weights.
\item Let $\lieg=\GG_2$. Then the action of $\lambda$ on $\lieg$ has at least five pairs of nonzero weights.
 \end{enumerate}
Clearly, if $\lieg$ has a subroot system of any type above and the missing root is in the subroot system, then any $\lambda$ has too many nonzero weights. 
This allows one to   eliminate $r\lieg$ for   all Lie algebras of rank at least three, except for $\liesl_4$. For all Lie algebras of rank 2 and for $\liesl_4$, any $\lambda$ has too many nonzero weights if one considers $3$ or more copies of $\lieg$. For example, any $\lambda$ for $\liesl_3$ has at least four pairs of nonzero weights.  Lemma \ref{lem:sl2} takes care of multiples of $\liesl_2$.  Hence we are reduced to two copies of the Lie algebras of types   $\lie{so}_5$ and $\liesl_4$.   

The most straightforward thing to do is to find all the slice representations of our candidates. We can use the techniques of \cite[Lemma 3.8]{SchCoregular}.  We call a slice representation of some $(V',G')$ \emph{maximal proper\/} if it is not a slice representation of any slice representation of $(V',G')$ except for $(V',G')$ itself. Consider the case of $2\lie{so}_5$. Any isotropy group of a nonzero closed orbit in  $2\lie{so}_5$ occurs in the slice representation of  the isotropy group of a nonzero closed orbit in $\lie{so}_5$.
Thus we have to consider the following groups and slice representations.

\begin{enumerate}
\addtocounter{enumi}{4}
\item The maximal torus $T$ acting on the root spaces (plus a trivial representation) where the positive root spaces are $\lieg_\alpha$, $\lieg_\beta$, $\lieg_{\alpha+\beta}$ and $\lieg_{\alpha+2\beta}$.
\item The group with finite cover $\SL_2\times\C^*$ and slice representation $2\liesl_2+ R_1\otimes\nu_1+R_1\otimes\nu_{-1}+\nu_2+\nu_{-2}+\theta$. 
\item The group with finite cover $\SL_2\times\C^*$ and slice representation $2\liesl_2+R_2\otimes\nu_1+R_2\otimes\nu_{-1}+\theta$.
\end{enumerate}
With one exception, the maximal proper slice representations in cases (6) and  (7) are (5) or a slice representation of (5) with isotropy group $\C^*$. The exception is the slice representation $(2\liesl_2+2R_1,\SL_2)$ in (6) corresponding to a stratum of codimension 7. Its unique  maximal proper slice representation   is again a slice representation  of (5). For (5) the maximal proper  slice representations   are   $(3\nu_1+3\nu_{-1},\C^*)$ and $(2\nu_1+2\nu_{-1}+\nu_{-2}+\nu_{-2},\C^*)$. The second case is not subprincipal and has the nontrivial  slice representation $(\C^4+\theta,\pm 1)$ of a codimension four stratum. There is no stratum of codimension five with a finite isotropy group. Our claims for the case of $2\lie{so}_5$ now follow easily. For $2\liesl_4$ one finds similarly that there are two subprincipal strata $S_1$ and $S_2$, corresponding to $(3\nu_1+3\nu_{-1},\C^*)$ and $(4\nu_1+4\nu_{-1},\C^*)$ and these are slice representations of the maximal torus. The maximal stratum in the closure of $S_1$  corresponds to a  $2$-torus and has codimension $8$. This concludes the proof for the case where $\codim S=5$ and $H^0=\C^*$.

If $\codim S=5$ and the isotropy group is finite, then by Corollary \ref{cor:finiteH} we have a stratum $S'$ of codimension 4 which is subprincipal. Thus we need only consider the case where $\codim S=4$ and $H$ is a finite subgroup of the maximal torus $T$ of $G$. Applying the roots to $H$ we see that we must have at most two positive roots $\gamma$ such that $\gamma$ is nontrivial on $H$. If there is only one simple root nontrivial on $H$, then   (1)--(4) of this proof show that we need only consider $2\liesl_4$, $2\lie{so}_5$, $2\liesl_3$ and $3\liesl_3$. For $r\liesl_3$   there are no finite nontrivial isotropy groups and for the other cases we have calculated that only $2\lie{so}_5$ gives us a stratum of codimension 4 with finite isotropy group. If two simple roots $\alpha$ and $\beta$ are nontrivial on $H$ and $G$ has rank at least three, then there is a third simple root $\gamma$ such that $\alpha+\gamma$ or $\beta+\gamma$ is a root giving  at least six roots nontrivial on $H$. So the rank is at most two. One easily sees that $\GG_2$ is not a possibility, so there are no new cases to consider. This completes the proof.
\end{proof}

Let $V=\oplus_{i=1}^d r_i\lieg_i$, $G$ and $Z=Z_1\times\dots\times Z_d$ be as in the beginning of this section. Let $\phi\in\Aut(Z)$. Then by  \cite{SchVectorFields} we know that $\phi$  permutes the strata of the same codimension. The strata of $Z$ are products $S_1\times\dots\times S_d$ where $S_i$ is a stratum of $Z_i$. If $S:=S_{i_1}\times\dots\times S_{i_r}$ is a product of strata where $S_{i_j}\subset Z_{i_j}$, we let $S'$ denote the product of $S$ with the principal strata of the remaining $Z_k$. We say that a stratum $S$ of $Z$ is \emph{simple\/} if it is $S_i'$ for some $S_i\subset Z_i$, otherwise we say that $S$ is \emph{composite}. By Corollary \ref{cor:codim2strata} the strata of codimension at most 5 are simple, hence they are permuted by $\phi$. 

\begin{lemma}
Let $S$ be a stratum of codimension $c$ at most $8$. Then $S$ is simple if and only if $\phi(S)$ is simple.
\end{lemma} 
 \begin{proof} 
 If $S$ is composite, then without loss of generality we have that $S=(S_1\times R_2)'$  where $S_1$ is a stratum of $Z_1$ and  $R_2$ is a stratum of $Z_2$. The codimensions of the strata add up to $c\geq 6$.  
Suppose that $\codim S_1=3$. Then for some $i$, $\phi(S_1')=S_i'$ where $\codim S_i=3$  in $Z_i$. Similarly, $\phi(R_2')=R_j'$ where $\codim R_j=c-3$.   If  $i=j$, then by Remark \ref{rem:codim3}   we would have that $R_j'$ is in the closure of $S_i'$. Applying $\phi\inv$ we would get that $R_2'$ is in the closure of $S_1'$ which is absurd. Hence $i\neq j$. Now $(S_1\times R_2)' $ is the maximal dimensional stratum in $\overline{S_1'}\cap\overline{R_2'}$ and similarly for $(S_i\times R_j)'$. It follows that $\phi$ sends $(S_1\times R_2)'$ onto $(S_i\times R_j)'$. Hence $S$ composite implies that $\phi(S)$ is composite. Using the same argument for $\phi\inv$ we get the converse. The remaining case is the product of two strata of codimension four which is handled similarly.
 \end{proof}

 \begin{corollary}\label{cor:preserveS}
 Let $\phi\in\Aut(Z)$. Then there is a $\sigma$-equivariant $\Psi\in\GL(V)$ inducing $\psi\in\Aut(Z)$ such that $\psi\inv\circ\phi$ preserves each stratum of codimension at most $5$. 
 \end{corollary}
 
 \begin{proof}  First suppose that $S=S_1'$ where $S_1$ has codimension $3$ and $r_1\lieg_1=2\liesl_3$. Then there is a stratum $R_1$ of codimension $4$ in $\overline{S_1}$. Now $\phi(S)=S_i'$ where $S_i$ has codimension $3$.  If $i\neq 1$ we may suppose that $i=2$. Then $\phi(R_1')$ is a simple stratum of codimension 4 which is in the closure of $S_2'$. The only possibility is that $\phi(R_1')$ is $R_2'$ where $R_2$ is a codimension 4 stratum of $Z_2$. It follows from Remark \ref{rem:codim3} that   $r_2\lieg_2=2\liesl_3$. Continuing inductively  and perhaps renumbering the factors of $\lieg$, we find that $r_i\lieg_i\simeq 2\liesl_3$ for $1\leq i\leq k$ where $\phi(S_i')=S_{i+1}'$ and $\phi(R_i')=R_{i+1}'$ for $1\leq i<k$ and $\phi(S_k')=S_1'$ and $\phi(R_k')=R_1'$. Here the $S_i$ have codimension 3 in $Z_i$ and the $R_i$ have codimension 4.  In case that $\phi(S_1')=(S_1')$ we have that $\phi$ also preserves  $R_1'$. There is an obvious $\sigma$-equivariant element $\Psi\in \GL(V)$ such that the induced mapping $\psi$ permutes  the $S_i'$ and $R_i'$ as $\phi$ does, $1\leq i\leq k$. Moreover, $\psi$ is the identity on all other strata of codimension at most $5$.  Thus $\psi\inv\circ\phi$ preserves each $S_i'$ and $R_i'$ for $1\leq i\leq k$. We could have that there is no codimension four stratum in $Z_1$, in which case we have $r_1\lieg_1\simeq 3\liesl_2$ and we continue as before. The 
 other possibilities for $S_1$ follow the same pattern using the information about the  strata of codimension at most $8$ in Remark \ref{rem:codim3} and Proposition  \ref{prop:classifS}. Thus we may reduce to the case that $\phi$ preserves each stratum $S_i'$ of codimension at most 5 for $1\leq i\leq k$. One then proceeds by the obvious induction.
 \end{proof}
 
Let $E$, $F$, $M$, $\E$, $\F$, $\G$, etc.\ be as in \S \ref{sec:kuttler}. Then we have condition  (*) and by \cite[Proposition 9.3, Theorem 10.7]{SchLifting} we also have (**).   
Recall that $V^\dag$ is the inverse image of the strata of codimension at most 5 and that $V\setminus V^\dag$ has codimension 4. Our aim now is to establish that 
\begin{equation}\label{eq:3star}\tag{***}
0\to\pi^\#\phi^*\E\to\pi^\#\phi^*\F\to\pi^\#\G\to 0
\end{equation}
is exact on  $V^\dag$ and that $\pi^\#\phi^*\E$ and $\pi^\#\phi^*\F$ are locally free on $V^\dag$. 
 Suppose that we have this. Then by  \cite[Remark 13]{Kuttler}, $\HHH^1(V^\dag,\pi^\#\phi^*\E)$ and $\HHH^1(V^\dag,\pi^\#\phi^*\F)$ are finitely generated. Taking cohomology in (***) over $V^\dag$ and using  the fact that $M$ has depth $3$ for the ideal of $V\setminus V^\dag$  it follows   that  \eqref{eq:secondglobal} is exact as well as the analogous result with $\phi$ replaced by $\phi\inv$. One concludes as before  that $\Gamma(V,\pi^\#\phi^*\F')$ is a free $\O(V)$-module. This implies that there is a lift $\Phi$ of $\phi$ over $V_\pr$, hence over $V$. 
 
     \begin{lemma}\label{lem:Mreflexive}
 The module $M$ is reflexive.
 \end{lemma}
 
 \begin{proof}
Suppose that  $M=M^{**}$ over an open subset $V^\flat\subset V$ whose complement has codimension 3. Then we see as before that $M$ has depth 2 with respect to the ideal of $V\setminus V^\flat$ so that $M\to M^{**}$ is an isomorphism. Set $V^\flat:=\{v\in V\mid \lieg_v= 0\}$. Over the open set $V^\flat$,  $M$ is  locally free, hence isomorphic to $M^{**}$, so  we need only prove that $V\setminus V^\flat$ has codimension at least 3. Now the inverse image of the strata of $Z$ of codimension at least four has codimension three in $V$, so we only need worry about the inverse image of the strata of codimension three. Using the slice theorem, it is enough to show that $W^\flat$ has codimension three in $W$ where $W$ is an $H=\C^*$ module such that $W'=2\nu_1+2\nu_{-1}$. But $\lieh$ only vanishes at the origin of $W'$.  
 \end{proof}

  \begin{proposition}\label{prop:Zariskiopen}
 Let $V$ be a $2$-principal $G$-module and $S$ a stratum of $Z$ with isotropy group $H$ where $H^0$ is a torus. Let   $v\in\pi\inv(S)$ and let $\phi\in\Aut(Z)$ such that $\phi$ preserves $S$.  
 \begin{enumerate}
 \item There is a neighborhood $N_z$ of $z:=\pi(v)\in Z$ (classical topology) and a biholomorphic lift $\Phi$ of $\phi$ sending $\pi\inv(N_z)$  onto  $\pi\inv(\phi(N_z))$.
 \item Let $v'=\Phi(v)$. Let $U$ (resp.\ $U'$) be a Zariski open neighborhood of $v$ (resp.\ $v'$).  Then   $\pi\inv(\phi(\pi(U)))$ (resp.\ $\pi\inv(\phi\inv(\pi(U'))))$ contains a neighborhood of $v'$ (resp.\ $v$) in the Zariski topology.   
 \end{enumerate}
  \end{proposition}
 
  \begin{proof}
Suppose that we have (1). Let $U$ be a neighborhood of $v$ in the Zariski topology. Then $\Phi(\pi\inv(N_z)\cap  U)\subset \pi\inv(\phi(N_z\cap  \pi( U)))$ so that $v'=\Phi(v)$ is an interior point of the constructible set $C:=\pi\inv(\phi(\pi(U)))$. Removing the boundary   we obtain a Zariski open neighborhood of $v'$ contained in $C$. Using $\Phi\inv$ we get the analogous property for neighborhoods of $v'$. Hence we only have to produce the local lift $\Phi$ of $\phi$. 

Since $\phi$ preserves $S$, the holomorphic slice theorem gives us an  equivariant biholomorphism $\Psi$ of $\pi\inv(N_{\phi(z)})$  with $\pi\inv(N_z)$  where  $N_z$ is a neighborhood of $z$ and   $N_{\phi(z)}$ is a neighborhood of $\phi(z)$. Let  $\psi\colon N_{\phi(z)}\to N_z$ be the induced map. We can arrange that $\psi(\phi(z))=z$.    By  Proposition \ref{prop:locallift} we can lift $\psi\circ\phi$ to a biholomorphic map over a    neighborhood $N_z'$ of $z$. Composing with $\Psi\inv$ we find a local lift of $\phi$ as required, and we have (1).
  \end{proof}

 \begin{proof}[Proof of Theorem \ref{thm:adjoint}] We can reduce  to the case that $\phi$ preserves each stratum of codimension at most 5. We show that (***) is exact on $V^\dag$.
Let  $\LL$ be one of $\O_Z$, $\E$, $\F$ or $\G$. Let $v\in V^\dag$ and define 
$$
\widetilde\LL_v=\lim_{U\ni v}\LL(\pi(V_\pr\cap U))
$$
where $U$ ranges over Zariski open subsets of $V$ containing $v$. Set $\O^G_v:=\widetilde{(\O_Z)}_v$. Then  $\widetilde\LL_v$ is an $\O_v^G$-module.
Recall that $i\colon V_\pr\to V$ is inclusion and that $\pi^\#\LL=i_*(\pi|_{V_\pr})^*\LL$.  Since $V\setminus V_\pr$ has codimension two in $V$, $i_*\O_{V_\pr}\simeq \O_V$.  Then $(\pi^\#\LL)_v$ is
$$
\O_{V,v}\otimes_{\O^G_v}(\lim_{U\ni v}\LL(\pi(V_\pr\cap U)))=\O_{V,v}\otimes_{\O^G_v}\widetilde\LL_v.
$$

Let $\Phi$ be as in Proposition \ref{prop:Zariskiopen} and set $v':=\Phi(v)$. Since $\phi$ preserves $Z_\pr$,  Proposition \ref{prop:Zariskiopen}(2) shows that pull-back by $\phi$ induces an isomorphism $\O^G_{v'}\to \O^G_v$ and
we have that
$(\widetilde{\phi^*\LL})_v\simeq \O_v^G\otimes_{\O_{v'}^G}\widetilde\LL_{v'}$.  Hence $(\pi^\#\phi^*\LL)_v
\simeq\O_{V,v}\otimes_{\O^G_{v'}}\widetilde\LL_{v'}$. Now we can tensor with the completions $\widehat{\O}_{V,v}$ of $\O_{V,v}$ and $\widehat{\O}_{V,v'}$ of $\O_{V,v'}$ where  $\Phi$ gives us an isomorphism $\hat\Phi\colon \widehat{\O}_{V,v'}\to\widehat{\O}_{V,v}$. We have completions
$$
(\widehat{\pi^\#\phi^*\LL})_v\simeq \widehat{\O}_{V,v}\otimes_{\O^G_{v'}}\widetilde\LL_{v'}\text{ and } (\widehat{\pi^\#\LL})_{v'}\simeq \widehat{\O}_{V,v'}\otimes_{\O^G_{v'}}\widetilde\LL_{v'}.
$$   
Hence  
$$
(\widehat{\pi^\#\phi^*\LL})_v\simeq \widehat{\O}_{V,v}\otimes_{\widehat{\O}_{V,v'}}(\widehat{\pi^\#\LL})_{v'}.
$$
Since completion is faithfully flat, we find that the sequence (***) is exact at $v$ if and only if the sequence
\begin{equation}\label{eq:lasteqn}
0\to\pi^\#\E\to\pi^\#\F\to\pi^\#\G\to 0
\end{equation}
is exact at $v'$ (recall that $\phi^*\G\simeq\G$).     As we saw in \S \ref{sec:kuttler},  Lemma \ref{lem:Mreflexive} implies that \eqref{eq:lasteqn} is exact, hence so is (***).  Finally, let $\LL$ be $\E$ or $\F$. Then $\pi^\#\LL$ is a free module, and as above one shows that  
the $\O_{V,v}$-module $(\pi^\#\phi^*\LL)_v$ has the same number of minimal generators as the $\O_{V,v'}$-module  $(\pi^\#\LL)_{v'}$. Thus $\pi^\#\phi^*\LL$ is locally free over $V^\dag$, and our proof is complete.
 \end{proof}

 \section{Actions of compact groups} \label{sec:compact}
 
 Some things become simpler when we are dealing with compact group actions. Let $K$ be a compact Lie group and $W$ a real $K$-module. Let $\pi\colon W\to W/K$ denote the quotient mapping. We can put a smooth structure on   $W/K$ (see \cite{SchSmooth}), as follows. A function on $W/K$ is smooth if it pulls back to a smooth (necessarily $K$-invariant) smooth function on $W$. We can make this more concrete, as follows. Let $p_1,\dots,p_d$ be homogeneous generators of $\R[W]^K$ and let $p=(p_1,\dots,p_d)\colon W\to\R^d$. Now $p$ is proper and separates the $K$-orbits in $W$. Let $X$ denote $p(W)\subset\R^d$. Then $X$ is a closed semialgebraic set and $p$ induces a homeomorphism $\bar p\colon W/K\to X$. The main theorem of \cite{SchSmooth} says that $p^*C^\infty(X)=C^\infty(W)^K$ where a function $f$ on $X$ is smooth if it extends to a smooth function in a neighborhood of $X$. Now we can define the notion of a smooth mapping $X\to X$ (or $W/K\to W/K$) in the obvious way. We also see an $\R^*$-action on $X$ where $t\cdot(y_1,\dots,y_d)=(t^{e_1}y_1,\dots,t^{e_d}y_d)$ for $t\in\R^*$, $(y_1,\dots,y_d)\in X$ and $e_j$   the degree of $p_j$, $j=1,\dots,d$. Or one can just say that $\R^*$ acts on $W/K\simeq X$ where $t\cdot p(w)=p(tw)$, $t\in \R^*$, $w\in W$. We denote by $\Aut_\ql(W/K)$ the quasilinear automorphisms of $W/K$, i.e., the automorphisms which commute with the $\R^*$-action. 
 
  We have the usual stratification of $W/K$ determined by the conjugacy classes of isotropy groups of orbits. There is a unique open  stratum, the principal stratum,  which consists of the orbits with trivial slice representation. It is connected and dense in $W/K$. The fixed points are, of course, the lowest dimensional stratum.
 
  The quotient $W/K$ has many nice properties, without any special conditions on $W$. From \cite[Theorem A]{BierstoneLifting} we have 
  
  \begin{theorem}
  Let $\phi$ be a diffeomorphism of $W/K$. Then $\phi$ preserves the stratification by the irreducible components of the isotropy type strata of $W/K$. In particular, $\phi$ preserves the principal stratum and the fixed points.
  \end{theorem}
  
 From  \cite[Theorem 2.2]{LosikLift} we have 
  
    \begin{theorem}
  Let $\phi\colon W/K\to W/K$ be a smooth isomorphism with $\phi(\pi(0))=\pi(0)$.   Then $\lim_{t\to 0}(t\inv\circ\phi\circ t)(x)$ exists for every $x\in W/K$ and the resulting map $\phi_0\colon W/K\to W/K$ is  a quasilinear isomorphism.
  \end{theorem}

 We have the complexification $G=K_\C$ acting on $V=W\otimes_\R\C$. Then our generators $p_i$ of $\R[W]^K$, considered as polynomials on $V$, generate $\O(V)^G$. 
 To obtain $\sigma$-equivariant lifts of automorphisms of $W/K$ we need to assume that $V$ is $2$-principal, in which case we say that $W$ is \emph{admissible}.
We say that an admissible $W$ has the \emph{lifting property\/} if every quasilinear isomorphism $\phi$ of $W/K$ has a  differentiable  lift $\Phi$ to  $W$.   Since $\Phi(0)=0$, one shows as before that $\Phi'(0)$ also lifts $\phi$ so if there is any lift, there is a linear one.

\begin{lemma}\label{lem:KtoG}
Let $\phi\in\Aut_\ql(W/K)$ where $W$ has the lifting property. Then there is a lift $\Phi\in N_{\GL(W)}(K)$.
\end{lemma}

\begin{proof}
There is a lift $\Phi\in\End(W)$ of $\phi$. There are canonical extensions $\widetilde\phi$ of $\phi$ to $\Aut_\ql(\quot VG)$ and $\widetilde\Phi$ of $\Phi$ to $\End(V)$ where $\widetilde\Phi$ is a lift of $\widetilde\phi$. By Theorem \ref{thm:stratapreserve}, $\widetilde\phi$ is strata preserving and Proposition  \ref{prop:normalizer} shows that $\widetilde{\Phi}$ is invertible and normalizes $G$. It follows that $\Phi$ is invertible and normalizes  $K=G\cap\GL(W)$.
\end{proof}

The following is a variant of \cite[Corollary 2.3]{LosikLift}.
 
 \begin{theorem}\label{thm:compactlifting}
Let $\phi\colon W/K\to W/K$ be a diffeomorphism where $W$ has the lifting property. Then there is a $\sigma$-equivariant diffeomorphism $\Phi$ of $W$ lifting $\phi$ for some $\sigma\in\Aut(K)$.
\end{theorem}  

\begin{proof} 
We may assume that $\phi(\pi(0))=\pi(0)$. We have an isotopy $\phi_t$ where $\phi_1=\phi$ and $\phi_0$ is quasilinear. But   Lemma \ref{lem:KtoG} there is a  lift $\Phi_0\in N_{\GL(W)}(K)$ of $\phi_0$. Thus it suffices to equivariantly lift $\phi_0\inv\circ\phi$. But this is a consequence of the isotopy lifting theorem of \cite{SchLifting}. 
\end{proof}

 \begin{proposition}\label{prop:cxlift}
If $V$ is $2$-principal and has the lifting property, then so does $W$.
\end{proposition}

\begin{proof}   Let $\phi\in\Aut_\ql(W/K)$. As above, $\phi$ extends to an element (which we also call $\phi$) of $\Aut_\ql(Z)$ which preserves the stratification. By hypothesis there is a lift $\Phi \in\GL(V)$ of $\phi$ and, by Proposition \ref{prop:normalizer}, $\Phi$ normalizes $G$. As in Remark \ref{rem:conjbyg}, we may change $\Phi $ by an arbitrary inner automorphism of $G$. Since $\Phi K\Phi \inv$ is conjugate to $K$ we can reduce to the case that $\Phi\in N_G(K)$.   Define $\overline{\Phi}$ by $\overline{\Phi}(v)=\overline{\Phi(\bar v))}$, $v\in V$. One similarly defines $\overline\phi$, but since $\overline\phi=\phi$ on $W/K$, we have that $\overline \phi=\phi$.
Thus $\overline \Phi $ covers $\phi$. Since $K$ is fixed by complex conjugation, $\overline \Phi $ also normalizes $K$. Then $\Phi \inv\overline \Phi $ normalizes $K$ and induces the identity on $Z$. By Proposition \ref{prop:normalizer}, $\Phi \inv \overline \Phi $ is multiplication by an element $g\in G$ and $g\in N_G(K)$. We have that
$$
\lieg=\liek\oplus i\liek=\liez(\liek)\oplus i\liez(\liek)\oplus [\liek,\liek]\oplus i[\liek,\liek]
$$
where $\liez(\liek)$ denotes the center of $\liek.$ Since no nonzero element of $i[\liek,\liek]$ can normalize $[\liek,\liek]$,  the Lie algebra of $N_G(K)$ is contained in $i\liez(\liek)\oplus\liek$. If $A\in \lie z(\liek)$ and $A$ is not fixed by $K$, then $\exp(iA)$ does not lie in $N_G(K)$. Hence the Lie algebra of $N_G(K)$ is $i\lie z(\liek)^K\oplus \liek$. Since $K$ is a maximal compact subgroup of $N_G(K)$,   our element $g$ can be written as $kp$ where $k\in K$, $p=\exp(iA)$   and $A\in \liez(\liek)^K$. Let $q=\exp(iA/2)$ and set $\Phi ':=\Phi q$. Then $\Phi '$ is still a lift of $\phi$ and $(\Phi ')\inv\overline{\Phi '}=q\inv kpq\inv=k$.  Thus $\overline{\Phi '}=\Phi 'k$.

We have reduced to the case that $\overline\Phi=\Phi k$ for some $k\in K$. Let $\widetilde W$ denote $\Phi(W)$. Then $\widetilde W$ is $K$-stable and $\pi(\widetilde W)=W/K$. Let $w\in W$. Then $\overline{\Phi(w)}=\Phi(kw)\in\widetilde W$ so that $\widetilde W$ is stable under complex conjugation.   Hence $\widetilde W=W_1+iW_2$ where $W_1$ and $W_2$ are $K$-stable subspaces of $W$. Since $\widetilde W\otimes_\R\C=V$, we must have that $W=W_1\oplus W_2$. 
Since $W_2$ has an invariant quadratic form $|\cdot|$, we have an invariant $q(w_1,w_2)=|w_2|^2$, which we may also consider as an invariant on $V$. Then $q$ induces a function  on $W/K$ which is nonnegative and  $q$ must be nonnegative on $\widetilde W$ since $\pi(\widetilde W) =W/K$. This gives a contradiction unless $W_2=0$. Hence $\Phi(W)=W$  and $\Phi$ is the desired lift of $\phi$.
\end{proof}

Theorem \ref{thm:compact} of the Introduction is a consequence of Theorem \ref{thm:compactlifting} and Proposition \ref{prop:cxlift}.

\smallskip
Many of the examples in section \S \ref{sec:examples} (Theorem \ref{thm:citlifting} and Remark \ref{rem:nolift}) apply to the case of representations of compact groups.  Let $\GG_2(\R)$ and $\BB_3(\R)$ denote the compact forms of $\GG_2$ and $\BB_3$ which act on $\R^7$ and $\R^8$, respectively. Then the representations $(4\R^7,\GG_2(\R))$, $(5\R^8,\BB_3(\R))$  and $(2\C^2,\SU_2(\C))$ do not have the lifting property. On the other hand, by Proposition \ref{prop:cxlift}, we have many examples $(W,K)$ which have the lifting property.

  

\begin{thebibliography}{KLM03}

\bibitem[Bie75]{BierstoneLifting}
Edward Bierstone, \emph{Lifting isotopies from orbit spaces}, Topology
  \textbf{14} (1975), no.~3, 245--252.

\bibitem[Got69]{GottschlingInvarianten}
Erhard Gottschling, \emph{Invarianten endlicher {G}ruppen und biholomorphe
  {A}bbildungen}, Invent. Math. \textbf{6} (1969), 315--326.

\bibitem[Hei91]{Heinzner91}
Peter Heinzner, \emph{Geometric invariant theory on {S}tein spaces}, Math. Ann.
  \textbf{289} (1991), no.~4, 631--662.

\bibitem[HK95]{HKOka}
Peter Heinzner and Frank Kutzschebauch, \emph{An equivariant version of
  {G}rauert's {O}ka principle}, Invent. Math. \textbf{119} (1995), no.~2,
  317--346.

\bibitem[Hum75]{HumphLinAlg}
James~E. Humphreys, \emph{Linear algebraic groups}, Springer-Verlag, New York,
  1975, Graduate Texts in Mathematics, No. 21.

\bibitem[KLM03]{KrieglTensor}
Andreas Kriegl, Mark Losik, and Peter~W. Michor, \emph{Tensor fields and
  connections on holomorphic orbit spaces of finite groups}, J. Lie Theory
  \textbf{13} (2003), no.~2, 519--534.

\bibitem[Kra84]{KraftBook}
Hanspeter Kraft, \emph{Geometrische {M}ethoden in der {I}nvariantentheorie},
  Aspects of Mathematics, D1, Friedr. Vieweg \& Sohn, Braunschweig, 1984.

\bibitem[Kur97]{Kurth}
Alexandre Kurth, \emph{{${\rm SL}_2$}-equivariant polynomial automorphisms of
  the binary forms}, Ann. Inst. Fourier (Grenoble) \textbf{47} (1997), no.~2,
  585--597.

\bibitem[Kut11]{Kuttler}
J.~Kuttler, \emph{Lifting automorphisms of generalized adjoint quotients},
  Transformation Groups \textbf{16} (2011), 1115--1135.

\bibitem[LMP03]{PopovMichor}
Mark Losik, Peter~W. Michor, and Vladimir~L. Popov, \emph{Invariant tensor
  fields and orbit varieties for finite algebraic transformation groups}, A
  tribute to {C}. {S}. {S}eshadri ({C}hennai, 2002), Trends Math.,
  Birkh{\"a}user, Basel, 2003, pp.~346--378.

\bibitem[Los01]{LosikLift}
M.~V. Losik, \emph{Lifts of diffeomorphisms of orbit spaces for representations
  of compact {L}ie groups}, Geom. Dedicata \textbf{88} (2001), no.~1-3, 21--36.

\bibitem[Lun73]{LunaSlice}
Domingo Luna, \emph{Slices \'etales}, Sur les groupes alg\'ebriques, Soc. Math.
  France, Paris, 1973, pp.~81--105. Bull. Soc. Math. France, Paris, M\'emoire
  33.

\bibitem[Pri67]{PrillLocal}
David Prill, \emph{Local classification of quotients of complex manifolds by
  discontinuous groups}, Duke Math. J. \textbf{34} (1967), 375--386.

\bibitem[Rag78]{Raghunathan}
M.~S. Raghunathan, \emph{Principal bundles on affine space}, C. {P}.
  {R}amanujam---a tribute, Tata Inst. Fund. Res. Studies in Math., vol.~8,
  Springer, Berlin, 1978, pp.~187--206.

\bibitem[Ros61]{Rosenlicht}
Maxwell Rosenlicht, \emph{Toroidal algebraic groups}, Proc. Amer. Math. Soc.
  \textbf{12} (1961), 984--988.

\bibitem[Sch75]{SchSmooth}
Gerald~W. Schwarz, \emph{Smooth functions invariant under the action of a
  compact {L}ie group}, Topology \textbf{14} (1975), 63--68.

\bibitem[Sch78]{SchCoregular}
\bysame, \emph{Representations of simple {L}ie groups with regular rings of
  invariants}, Invent. Math. \textbf{49} (1978), no.~2, 167--191.

\bibitem[Sch80]{SchLifting}
\bysame, \emph{Lifting smooth homotopies of orbit spaces}, Inst. Hautes
  \'Etudes Sci. Publ. Math. (1980), no.~51, 37--135.

\bibitem[Sch87]{SchCIT}
\bysame, \emph{On classical invariant theory and binary cubics}, Ann. Inst.
  Fourier (Grenoble) \textbf{37} (1987), no.~3, 191--216.

\bibitem[Sch88]{SchG2}
\bysame, \emph{Invariant theory of {$G\sb 2$} and {${\rm Spin}\sb 7$}},
  Comment. Math. Helv. \textbf{63} (1988), no.~4, 624--663.

\bibitem[Sch89]{SchTopAlgQuots}
\bysame, \emph{The topology of algebraic quotients}, Topological methods in
  algebraic transformation groups (New Brunswick, NJ, 1988), Progr. Math.,
  vol.~80, Birkh\"auser Boston, Boston, MA, 1989, pp.~135--151.

\bibitem[Sch95]{SchLiftingDOs}
\bysame, \emph{Lifting differential operators from orbit spaces}, Ann. Sci.
  \'Ecole Norm. Sup. (4) \textbf{28} (1995), no.~3, 253--305.

\bibitem[Sch13a]{SchAdjoint}
\bysame, \emph{Lifting automorphisms of quotients of adjoint representations},
  http://arxiv.org/abs/1201.6369 (2013).

\bibitem[Sch13b]{SchVectorFields}
\bysame, \emph{Vector fields and {L}una strata}, J. Pure and Applied Algebra
  \textbf{217} (2013), 54--58.

\bibitem[Str82]{StrubLocal}
Rainer Strub, \emph{Local classification of quotients of smooth manifolds by
  discontinuous groups}, Math. Z. \textbf{179} (1982), no.~1, 43--57.

\bibitem[VP89]{PopovVinberg}
{\`E}.~B. Vinberg and V.~L. Popov, \emph{Invariant theory}, Algebraic geometry,
  4 ({R}ussian), Itogi Nauki i Tekhniki, Akad. Nauk SSSR Vsesoyuz. Inst.
  Nauchn. i Tekhn. Inform., Moscow, 1989, pp.~137--314, 315.

\end{thebibliography}
 \def\cprime{$'$}
\providecommand{\bysame}{\leavevmode\hbox to3em{\hrulefill}\thinspace}
\providecommand{\MR}{\relax\ifhmode\unskip\space\fi MR }
\providecommand{\MRhref}[2]{%
  \href{http://www.ams.org/mathscinet-getitem?mr=#1}{#2}
}
\providecommand{\href}[2]{#2}

   \end{document}